\documentclass{amsart}
\usepackage{amssymb}

\usepackage{enumitem}

\usepackage{graphicx}

\makeatletter
\@namedef{subjclassname@2010}{%
	\textup{2010} Mathematics Subject Classification}
\makeatother

\usepackage[T1]{fontenc}
\newtheorem{thm}{Theorem}[section]
\newtheorem{lem}[thm]{Lemma}
\newtheorem{cor}[thm]{Corollary}
\newtheorem{prop}[thm]{Proposition}

\theoremstyle{definition}
\newtheorem{defn}[thm]{Definition}
\newtheorem{rmk}[thm]{Remark}

\numberwithin{equation}{section}

\newcommand{\ww}{{\mathbf{w}}}
\newcommand{\vvv}{{\mathbf{v}}}
\newcommand{\vv}[1][n]{{\mathbf{v}_{#1}}}
\newcommand{\dd}{{\mathbf{d}}}
\newcommand{\ee}{{\mathbf{e}}}
\newcommand{\kk}{{\mathbf{k}}}
\newcommand{\xx}{{\mathbf{x}}}
\newcommand{\zz}{{\mathbf{z}}}
\newcommand{\ra}{{\rightarrow}}
\newcommand{\nra}{{\nrightarrow}}

\newcommand{\z}{{\mathbb Z}}

\newcommand{\Mod}[1]{\ (\mathrm{mod}\ #1)}

\begin{document}

\title{On a quadratic Waring's problem with congruence conditions}

% Remove or comment out any unused author tags.
% author one information

\author[D. Kim]{Daejun Kim}
\address{Department of Mathematical Sciences\\ Seoul National University\\
 Seoul 151-747, Korea}
\email{goodkdj@snu.ac.kr}

\thanks{This work was supported by BK21 PLUS SNU Mathematical Science Division.}

% Use this \subjclass if you are using amsart version 2.0 (December 1999).
%\subjclass[2000]{11E12, 11E20}
% Use this one if you are using an older version of amsart.
\subjclass[2010]{Primary 11E12, 11E25}
%\date{}

\keywords{Waring's problem, Sums of squares, Representations of cosets}

%\dedicatory{}

% at present the "communicated by" line appears only in ERA, PROC and JAG
%\commby{}

\begin{abstract} 
	For each positive integer $n$, let $g_\Delta(n)$ be the smallest positive integer $g$ such that every complete quadratic polynomial in $n$ variables which can be represented by a sum of odd squares is represented by a sum of at most $g$ odd squares.
	In this paper, we analyze $g_\Delta(n)$ by studying representations of integral quadratic forms by sums of squares with certain congruence condition.
	We prove that the growth of $g_\Delta(n)$ is at most an exponential of $\sqrt{n}$, which is the same as the best known upper bound on the $g$-invariants of the original quadratic Waring's problem.
	We also determine the exact value of $g_\Delta(n)$ for each positive integer less than or equal to $4$.
\end{abstract}

\maketitle

\section{Introduction}
In 1770, Lagrange proved the Four-Square Theorem, which states that every positive integer is a sum of at most four squares of integers. This result has been generalized in many directions. 
In 1930's, a higher dimensional generalization, the so-called {\em new} ({\em or quadratic}) {\em Waring's Problem}, was initiated and studied by Mordell \cite{M1} and Ko \cite{K1}. 
In those papers, they proved that for any integer $1\le n \le 5$, every positive definite integral quadratic form in $n$ variables is represented by a sum of $n+3$ squares, and $n+3$ is the smallest number with this property. 
Later, Mordell \cite{M2} proved that the quadratic form corresponding to the root lattice $E_6$ cannot be represented by any sum of squares. 

This result of Mordell lead us to consider the number $g_\mathbb{Z}(n)$ defined as the smallest positive integer $g$ such that every quadratic form in $n$ variables which can be represented by a sum of squares is represented by a sum of at most $g$ squares.
The numbers $g_\mathbb{Z}(n)$ are called the ``$g$-invariants" of $\z$.
Then the results of Lagrange's, Mordell's and Ko's mentioned above can now be rewritten as $g_\z(n)=n+3$ for $1\le n \le 5$. In \cite{KO1}, Kim and Oh proved that $g_\z(6)=10$, which disproves the earlier conjecture made by Ko \cite{K2} that $g_\z(6)=9$. This is the last known value of $g_\z(n)$.

On the other hand, it has been studied to find an upper bound of $g_\z(n)$ as a function of $n$.
Icaza \cite{I} gave the first explicit but astronomical upper bound by computing the so called HKK-constant in \cite{HKK}. 
Later, Kim-Oh \cite{KO3} proved that $g_\z(n)=O(3^{n/2}n\log n)$, which improves Icaza's bound.
Recently, Beli-Chan-Icaza-Liu \cite{BCIL} obtained a better upper bound $g_\z(n)=O(e^{(4+2\sqrt2+\varepsilon)\sqrt{n}})$ for any $\varepsilon>0$.

In this paper, we consider a quadratic Waring's problem with a congruence condition modulo $2$ as a generalization of the original problem.
One may naturally generalize the Lagrange's four square theorem by considering the smallest number $r$ such that every positive integer is a sum of at most $r$ squares of odd integers.
In fact, this number is $10$ and 
\[42=5^2+3^2+1^2+1^2+1^2+1^2+1^2+1^2+1^2+1^2\]
is the smallest positive integer which is a sum of $10$ squares of odd integers but is not a sum of less than $10$ squares of odd integers (cf. Proposition \ref{gdelta1}, see also \cite{KO}).

As a higher dimensional generalization of this problem, we introduce new $g$-invariants $g_\Delta(n)$ in the following paragraphs.

Let $f(\mathbf{x})=f(x_1,...,x_n)$ be a quadratic polynomial such that
\[f(\mathbf{x})= Q(\mathbf{x})+L(\mathbf{x})+c,\]
where $Q(\mathbf{x})$ is a quadratic form, $L(\mathbf{x})$ is a linear form, and $c$ is a constant. We always assume that $Q$ is positive definite. 
Hence, there exists a unique vector $\ww_f \in\mathbb{Q}^n$ such that $L(\mathbf{x})=2B(\mathbf{x},\ww_f)$, where $B$ is the bilinear form such that $B(\mathbf{x},\mathbf{x})=Q(\mathbf{x})$.
The quadratic polynomial $f(\mathbf{x})$ is called {\it complete} if $c=Q(\ww_f)$, that is,
\[f(\mathbf{x})= Q(\mathbf{x})+2B(\mathbf{x},\ww_f)+Q(\ww_f)= Q(\mathbf{x}+\ww_f).\]
We say that a quadratic polynomial $f(\mathbf{x})=f(x_1,...,x_n)$ is represented by a quadratic polynomial $g(\mathbf{y})=g(y_1,...,y_m)$ $(m\ge n)$ if there exists $T\in M_{n\times m}(\mathbb{Z})$ and $\mathbf{c}\in\mathbb{Z}^n$ such that  
\[f(\mathbf{x}) = g(\mathbf{x}T+\mathbf{c}).\]

Now let $\Delta_r(\mathbf{y})$ be the following quadratic polynomial in variables $\mathbf{y}=(y_1,...,y_r)$:
\[\Delta_r(y_1, ... , y_r ) :=  (2y_1+ 1)^2 + \cdots + (2y_r+ 1)^2.\]
A quadratic polynomial $f(\mathbf{x})$ is said to be represented by a sum of $r$ odd squares if it is represented by $\Delta_r(\mathbf{y})$. For each positive integer $n$, we define the set $\mathcal{F}_n$ of all complete quadratic polynomials $f(\mathbf{x})$ in variables $\mathbf{x}=(x_1,...,x_n)$ which can be represented by a sum of odd squares.
For a quadratic polynomial $f(\mathbf{x})$ in $\mathcal{F}_n$, we define
\[r(f):=\text{min}\left\{ r \in\mathbb{N} : f(\mathbf{x}) \text{ can be represented by }\Delta_r(\mathbf{y}) \right\},\]
and we define the following new $g$-invariant of $\Delta_r$:
\[g_\Delta(n):= \text{max}\left\{ r(f) : f(\mathbf{x}) \in \mathcal{F}_n \right\}. \]
One may deduce that the problem of determining $g_\Delta(1)$ is equivalent to the problem of representing positive integers by sums of odd squares explained above, so that $g_\Delta(1)=10$. Furthermore, we will see in Section \ref{geometricapproach} that $g_\Delta(n)$ can be analyzed by studying representation of integral quadratic forms by sums of squares with a congruence condition modulo $2$.
%We will see in Section \ref{geometricapproach} that $g_\Delta(n)$ can be analyzed by concerning representation of integral quadratic forms by sums of squares with a congruence condition modulo $2$.
%In particular, one may observe that the problem of finding $g_\Delta(1)$ is equivalent to the problem of concerning a representation of positive integers by sums of odd squares, hence $g_\Delta(1)=10$.
%Indeed, the complete quadratic polynomial 
%\[f(x)=42\cdot(2x+1)^2=\Delta_r(5x+2,3x+1,2x\underbrace{+1,\ ... \ , 2x+}_{8-\text{copies}}1)\]
%%\[f(x)=42\cdot(2x+1)^2=(2(5x+2)+1)^2+(2(3x+1)+1)^2+\underbrace{(2x+1)^2+\cdots+(2x+1)^2}_{8-\text{copies}}\]
Our main results can be stated as

\begin{thm} \label{thm1}
Let $n$ be a positive integer. For any $\varepsilon>0$ we have
\[g_\Delta(n) = O\left(e^{(4+2\sqrt{2}+\varepsilon)\sqrt{n}}\right).\]
\end{thm}

\begin{thm}\label{thm2}
We have $g_\Delta(1)=10$ and $g_\Delta(n)=n+10$ for $n=2,3,4$.
\end{thm}

Note that our result presents the same growth as the best known upper bound on $g_\z(n)$.
More precisely, the upper bound on $g_\Delta(n)$ we obtain is approximately $n^2$ times the upper bound on $g_\mathbb{Z}(n)$ obtained in \cite{BCIL}. 
We will adopt geometric language of quadratic spaces, lattices and $\mathbb{Z}$-cosets in studying $g_\Delta(n)$ so that we shall use the geometric theory of those.

The rest of the paper is organized as follows. 
In Section \ref{preliminaries}, we introduce the geometric language and theory of quadratic spaces, lattices and $\mathbb{Z}$-cosets, especially the concept of representations of $\z$-cosets. 
In Section \ref{geometricapproach}, we consider the problem geometrically by translating representations of quadratic polynomials into representations of $\z$-cosets explicitly. The exact value of $g_\Delta(1)$ will also be determined.
Section \ref{Lemmas} contains some technical lemmas which will essentially be used in the following sections.
The proof of Theorem \ref{thm1} will be presented in Section \ref{ubdsec}.
In Section \ref{gdelta1234}, we will determine the exact values of $g_\Delta(n)$ for $n=2,3,4$ through some extensive computation.

For any unexplained notations, terminologies, and basic facts about $\z$-lattices, we refer the readers to \cite{OM2}.

%%%%%%%%%%%%%%%%%%%%%%%%%%%%%%%%%%%%%%%%%%%%%%%%%%%%%%%%%%%%%%%%%%%
\section{Representation of cosets} \label{preliminaries}
In this section, we introduce the geometric theory of quadratic $\mathbb{Z}$-lattices. 
We refer the readers to \cite[Section 4]{CO} for the theory under more general setting. For simplicity, the quadratic map and its associated bilinear form on any quadratic space will be denoted by $Q$ and $B$, respectively. 
The set of all places on $\mathbb{Q}$ including the infinite place $\infty$ will be denoted by $\Omega$.

A $\mathbb{Z}$-lattice is a finitely generated $\mathbb{Z}$-module (hence a free $\mathbb{Z}$-module) $L$ on an $m$-dimensional quadratic space $V$ over $\mathbb{Q}$.
A $\mathbb{Z}$-coset is a set $L+\vvv$, where $L$ is a $\mathbb{Z}$-lattice on $V$ and $\vvv$ is a vector in $V$. 
A $\mathbb{Z}$-coset $K+\ww$ on an $n$-dimensional quadratic space $W$ is said to be represented by another $\mathbb{Z}$-coset $L+\vvv$  on an $m$-dimensional space $V$, which is denoted by 
\[K+\ww \ra L+\vvv,\]
if there exists an isometry $\sigma: W \ra V$ such that $\sigma(K+\ww) \subseteq L+\vvv$, which is equivalent to
\[ \sigma(K)\subseteq L \quad \text{and} \quad \sigma(\ww)-\vvv \in L.\]
Two $\mathbb{Z}$-cosets $K+\ww$ and $L+\vvv$ are said to be isometric, which is denoted by $K+\ww \cong L+\vvv$, if one is represented by another one and vice versa. 
For each $p\in\Omega-\{\infty\}$, $\mathbb{Z}_p$-cosets and representations of $\mathbb{Z}_p$-cosets are defined analogously.

As in the case of quadratic forms and lattices, there is a one-to-one correspondence between the set of equivalence classes of complete quadratic polynomials in $n$ variables and the set of isometry classes of $\mathbb{Z}$-cosets on $n$-dimensional quadratic spaces. We will describe this correspondence concretely in Proposition \ref{correspondence}.

\begin{defn}
Let $L+\vvv$ be a $\mathbb{Z}$-coset on a quadratic space $V$. The genus of $L+\vvv$ is the set 
\[\text{gen}(L+\vvv)=\{K+\ww \text{ on } V \text{ : } K_p+\ww \cong L_p+\vvv \text{ for any } p \in \Omega \}. \]
\end{defn}

\begin{lem}\label{finreplacelem}
Let $L+\vvv$ be a $\mathbb{Z}$-coset on a quadratic space $V$ and let $S$ be a finite subset of $\Omega$. Suppose that $\mathbb{Z}_p$-coset $L(p)+\xx_p$ on $V_p$ is given for each $p\in S$. Then there exists a $\mathbb{Z}$-coset $M+\zz$ on $V$ such that 
\[
M_p+\zz=
\begin{cases} 
L(p)+\xx_p & \text{if } p\in S,\\
L_p+\vvv & \text{if } p \in \Omega -S.
\end{cases}
\]
\end{lem}
\begin{proof}
See Lemma 4.2 of \cite{CO}.
\end{proof}

Let $O_\mathbb{A}(V)$ be the adelization of the orthogonal group of $V$. 
By Lemma \ref{finreplacelem}, $O_\mathbb{A}(V)$ naturally acts transitively on $\text{gen}(L+\vvv)$ and hence
\[\text{gen}(L+\vvv) = O_\mathbb{A}(V)\cdot (L+\vvv).\]
Let $O_\mathbb{A}(L+\vvv)$ be the stabilizer of $L+\vvv$ in $O_\mathbb{A}(V)$. Then the isometry classes in $\text{gen}(L+\vvv)$ can be identified with 
\[O(V)\backslash O_\mathbb{A}(V) / O_\mathbb{A}(L+\vvv).\]
The {\em class number} of $L+\vvv$, denoted by $h(L+\vvv)$, is the number of classes in $\text{gen}(L+\vvv)$, which is also the number of elements in $O(V)\backslash O_\mathbb{A}(V) / O_\mathbb{A}(L+\vvv)$.
The class number $h(L+\vvv)$ is finite and $h(L+\vvv) \ge h(L)$, where $h(L)$ is the class number of $L$ (see Corollary 4.4 of \cite{CO}). Note that $h(L)$ is equal to the number of elements in $O(V)\backslash O_\mathbb{A}(V)/O_\mathbb{A}(L)$. For each $p\in \Omega$, we have 
\[O(L_p+\vvv) = \{ \sigma\in O(V_p) \text{ : } \sigma(L_p)=L_p \text{ and } \sigma(\vvv) \equiv \vvv \text{ mod }L_p \} \subseteq O(L_p).\]

From now on, let $I_n=\mathbb{Z}[\ee_1,...,\ee_n]$ be the $\mathbb{Z}$-lattice whose Gram matrix with respect to $\{\ee_1,...,\ee_n\}$ is the identity matrix.
For the sake of convenience, the vector $\ee_1+\cdots +\ee_n$ will be denoted by $\vv$ and the $\mathbb{Z}$-coset $I_n + \frac{1}{2}\vv$ will be denoted by $\Sigma_n$.

\begin{prop} \label{hIn}
For any $1 \le n \le 8$, we have $h(\Sigma_n) = 1$.
\end{prop}
\begin{proof}
If we can prove for any prime $p$ that 
\[O((\Sigma_n)_p)=O((I_n)_p),\] 
then $O_\mathbb{A}(\Sigma_n)=O_\mathbb{A}(I_n)$ so that we have $h(\Sigma_n) = h(I_n) = 1$, which proves the proposition.

When $p\neq 2$, we have $(\Sigma_n)_p=(I_n)_p+\frac{1}{2}\vv = (I_n)_p$ so that $O((\Sigma_n)_p)=O((I_n)_p)$. 
Now, it suffices to show that $O((I_n)_2)\subseteq O((\Sigma_n)_2)$.
Let $\sigma\in O((I_n)_2)$ and for each $1\le i \le n$, we put $\sigma(\ee_i) = \sum_{k=1}^n t_{ik}\ee_k$ for some $t_{ik} \in \mathbb{Z}_2$. 
Note that 
\[
\sum_{k=1}^n t_{ik}^2=Q(\sigma(\ee_i))=1 \quad \text{and} \quad 
\sum_{k=1}^n t_{ik}t_{jk}=B(\sigma(\ee_i),\sigma(\ee_j))=0,
\]
for any $1\le i \neq j \le n$, since $\sigma\in O((I_n)_2)$.
We claim that $\sum_{i=1}^n t_{ik} \in 1+2\mathbb{Z}_2$ for any $1\le k \le n$ if the following two conditions hold:\\ [3pt]
\indent (1) $\sum_{k=1}^n t_{ik}^2\in 1+4\mathbb{Z}_2$ for any $1\le i \le n$, \\
\indent (2) $\sum_{k=1}^n t_{ik}t_{jk}\in 2\mathbb{Z}_2$ for any $1\le i \neq j \le n$.\\[3pt]
\noindent If we show this claim, then we have
\[
\sigma\left(\frac{1}{2}\vv\right) = \sum_{i=1}^n \frac{1}{2} \sigma(\ee_i) = \sum_{k=1}^n \left(\frac{1}{2}\cdot\sum_{i=1}^n t_{ik} \right) \ee_k \in (\Sigma_n)_2,
\]
which implies $\sigma \in O((\Sigma_n)_2)$ and therefore we prove the proposition.
%To complete the proof, we claim, for each $1\le n \le 8$, that if $\sum_{k=1}^n t_{ik}^2\in 1+4\mathbb{Z}_2$ and $\sum_{k=1}^n t_{ik}t_{jk}\in 2\mathbb{Z}_2$ for any $1\le i \neq j \le n$, then $\sum_{i=1}^n t_{ik} \in 1+2\mathbb{Z}_2$ for any $1\le k \le n$.

We prove the claim using an induction argument on $n$.
When $n=1$, we have $t_{11}=\pm 1 \in 1+2\mathbb{Z}_2$.
Now, assume that $n>1$ and the above two conditions hold.
From the first condition, for each $1\le i \le n$, exactly one or five of $\{ t_{i1},...,t_{in} \}$ belong to $1+2\mathbb{Z}_2$ and all the other elements are in $2\mathbb{Z}_2$.
Assume, without loss of generality, that $t_{nn}\in 1+2\mathbb{Z}_2$ and $t_{nk}\in2\mathbb{Z}_2$ for all $k<n$.
Then, from the second condition, $t_{in}\in 2\mathbb{Z}_2$ for any $i<n$.
Hence, we have $\sum_{i=1}^n t_{in}\in1+2\mathbb{Z}_2$. 
Also, we have $\sum_{k=1}^{n-1} t_{ik}^2\in 1+4\mathbb{Z}_2$ and $\sum_{k=1}^{n-1} t_{ik}t_{jk}\in 2\mathbb{Z}_2$ for any $1\le i \neq j \le n-1$ 
so that we have $\sum_{i=1}^n t_{ik}\in1+2\mathbb{Z}_2$ by the induction hypothesis.
Therefore, we are left with the case when $n\ge 5$ and exactly five of $\{t_{i1},...,t_{in}\}$ belong to $1+2\mathbb{Z}_2$ for any $1\le i \le n$. 
One may easily show from the second condition that this can only happen when $n=6,8$ as well as $\sum_{i=1}^n t_{ik}\in1+2\mathbb{Z}_2$ for any $1\le k \le n$ in those cases.
\end{proof}

\begin{prop} \label{repbygen}
Let $K+\ww$ be a $\mathbb{Z}$-coset on a quadratic space $W$, and let $L+\vvv$ be a $\mathbb{Z}$-coset on a quadratic space $V$. Suppose that for each $p\in \Omega$, there exists a representation $\sigma_p : W_p \ra V_p$ such that $\sigma_p(K_p+\ww)\subseteq L_p+\vvv$. Then there exists a $\mathbb{Z}$-coset $M+\zz\in \text{gen}(L+\vvv)$ which represents $K+\ww$.
\end{prop}
\begin{proof}
By virtue of the Hasse Principle, we may assume that $W\subseteq V$. 
By Witt's extension theorem, we may further assume that $\sigma_p\in O(V_p)$. 
Let $S$ be the set of $p\in\Omega$ such that $K_p+\ww \not \subseteq  L_p+\vvv$. 
Then $S$ is a finite set since $K_p+\ww = K_p \subseteq L_p = L_p+\vvv$ for almost all $p$. 
For each $p\in S$, let $L(p)=\sigma^{-1}_p(L_p)$ and $\xx_p = \sigma^{-1}_p(\vvv)$. 
By Lemma \ref{finreplacelem}, there exist $M+\zz\in \text{gen}(L+\vvv)$ such that 
\[
M_p+\zz=
\begin{cases} 
\sigma^{-1}_p(L_p+\vvv) & \text{if } p\in S,\\
L_p+\vvv & \text{if } p \in \Omega - S.
\end{cases}
\]
Therefore $K+\ww \subseteq M+\zz$, which proves the proposition.
\end{proof}

\begin{cor}\label{localglobal}
Let $K+\ww$ be a $\mathbb{Z}$-coset and let $n$ be a positive integer less than or equal to $8$. If $K+\ww$ is locally represented by $\Sigma_n$, then $K+\ww$ is represented by $\Sigma_n$.
\end{cor}
\begin{proof}
This is a direct consequence of Propositions \ref{hIn} and \ref{repbygen}.
\end{proof}

\section{Geometric approach of the problem} \label{geometricapproach}

In this section, we introduce some geometric approach of the problem via representations of $\mathbb{Z}$-cosets. 
For any $r\in\mathbb{N}$, let $I_r=\mathbb{Z} [\ee_1,...,\ee_r]$ be the $\mathbb{Z}$-lattice whose Gram matrix with respect to $\{\ee_1,...,\ee_r\}$ is the identity matrix.
As in Section \ref{preliminaries}, the vector $\ee_1+\cdots +\ee_r$ will be denoted by $\vv[r]$ and the $\mathbb{Z}$-coset $I_r + \frac{1}{2}\vv[r]$ will be denoted by $\Sigma_r$.
% On the other hand, we can view this problem as a reprentation of lattice cosets. In a nutshell, a complete quadratic polynomial $f(\mathbf{x})=Q(\mathbf{x}+\mathbf{w})$ with $\mathbf{w}=(w_1,...w_n)\in\mathbb{Q}^n$ is just a coset $K+\ww$ of a $\mathbb{Z}$-lattice $K$ on a quadratic $\mathbb{Q}$-space with a quadratic map $Q$, where $K=\mathbb{Z}[\dd_1,...,\dd_n]$ and $\ww=w_1\dd_1+\cdots+w_n\dd_n$. In particular, $\Delta_r(\mathbf{y})$ correspond to the coset $I_r+\frac{1}{2}\vv$ of $\mathbb{Z}$-lattice $I_r=\mathbb{Z} [\ee_1,...,\ee_r]\cong \langle 1,...,1\rangle$, where $\vv =  \ee_1+\cdots +\ee_r$.\\
For any positive integer $n$, we define
\[\mathcal{K}_n := \left\{ K+\ww \mid \text{rank} (K)=n, \ \ww\in\mathbb{Q}K, \ \exists \sigma : K+\ww\ra \Sigma_r \right\}.\]
For any $K+\ww\in\mathcal{K}_n$, we define
\[g(K+\ww):=\text{min} \left\{r\in\mathbb{N} \mid  \exists \sigma : K+\ww\ra \Sigma_r \right\},\]
and we also define
\[g_\Delta'(n):= \text{max}\left\{g(K+\ww) \mid K+\ww \in \mathcal{K}_n\right\}.\]

\begin{prop}\label{correspondence}
For any positive integer $n$, we have $g_\Delta(n)=g_\Delta'(n)$ 
\end{prop}
\begin{proof}
Let $f(\mathbf{x})=Q(\mathbf{x}+\ww_f)=(\mathbf{x}+\ww_f)4M(\mathbf{x}+\ww_f)^t$ be a complete quadratic polynomial in $\mathcal{F}_n$, where $\ww_f=(w_1,...,w_n)\in\mathbb{Q}^n$, and $4M$ is the Gram matrix of the quadratic part $Q$ of $f$. 
Hence, there exists a positive integer $r:=r(f)$, a matrix $T=(t_{ij})\in M_{n\times r}(\mathbb{Z})$, and a vector $\mathbf{c}=(c_1,...,c_n)\in\mathbb{Z}^n$ such that 
\[(\mathbf{x}+\ww_f)(4M)(\mathbf{x}+\ww_f)^t=f(\mathbf{x})=\Delta_r(\mathbf{x}T+\mathbf{c}).\]
By comparing the coefficients of both sides and by putting $\mathbf{x}=-\ww_f$, one may easily show that 
\begin{equation}\label{eq1}
M=T\cdot T^t \text{ and } -\ww_fT+\mathbf{c}+\frac{1}{2}(1,...,1)=\mathbf{0}\in \mathbb{Q}^r.
\end{equation}
Now let us consider a $\mathbb{Z}$-coset $K+\ww$ of a $\mathbb{Z}$-lattice $K=\mathbb{Z}[\dd_1,...,\dd_n]\cong M$, where $\ww=w_1\dd_1+\cdots+w_n\dd_n$ and 
define a linear map $\sigma : K\ra I_r$ by
\[\sigma(\dd_i) = \sum_{j=1}^r t_{ij}\ee_j.\]
By \eqref{eq1}, the map $\sigma$ is a representation of $\mathbb{Z}$-lattices satisfying $\sigma(\ww)-\frac{1}{2}\vv[r]\in I_r$, which implies that $\sigma : K+\ww \ra \Sigma_r$ is a representation of $\mathbb{Z}$-cosets. Thus, we have constructed a $\mathbb{Z}$-coset $K+\ww$ in $\mathcal{K}_n$ with $g(K+\ww)\le r=r(f)$.

Conversely, let $K+\ww$ be a $\mathbb{Z}$-coset in $\mathcal{K}_n$, where $K=\mathbb{Z}[\dd_1,...,\dd_n]$ and $\ww=w_1\dd_1+\cdots+w_n\dd_n$. 
Then there exist $g:=g(K+\ww)\in \mathbb{N}$ and a representation of $\mathbb{Z}$-cosets $\sigma : K+\ww \ra \Sigma_g$. 
Since $\sigma(\ww)-\frac{1}{2}\vv[g]\in I_g$, there are integers $c_1,...,c_g$ such that $\sigma(\ww)=\frac{1}{2}\vv[g]+c_1\ee_1+\cdots + c_g\ee_g$.
Also, let $T=(t_{ij})\in M_{n\times g}(\mathbb{Z})$ be the matrix such that $\sigma(\dd_i)=\sum_{j=1}^g t_{ij}\ee_j$ for each $1\le i \le n$. Then we have
\begin{equation}\label{eq2}
\begin{aligned}[t]
f(\mathbf{x})&:=4\cdot Q(x_1\dd_1+\cdots+x_n\dd_n+\ww)\\
&\, =4\cdot Q(\sigma(x_1\dd_1+\cdots+x_n\dd_n+\ww))\\
&\, =4\cdot Q\left([(x_1,...,x_n)T + (c_1,...,c_g)] \cdot (\ee_1,...,\ee_g)^t +\frac{1}{2}\vv[g] \right)\\
&\, =\Delta_g\left((x_1,...,x_n)T+(c_1,...,c_g)\right),
\end{aligned}
\end{equation}
where $\mathbf{x}=(x_1,...,x_n)$. Hence the complete quadratic polynomial $f(\mathbf{x})$ is represented by $\Delta_g(\mathbf{y})$. Therefore, we have constructed a quadratic polynomial $f(\mathbf{x})$ in $\mathcal{F}_n$ with $r(f)\le g=g(K+\ww)$. The proposition follows as a consequence.
\end{proof}

\begin{prop}
For any positive integer $n$, let
\[\mathcal{K}_n^* := \left\{ K+\ww \in \mathcal{K}_n \mid \ww = \frac{1}{2}\kk, \ \kk \text{ is a primitive vector of } K\right\}.\] 
Then we have 
\[g_\Delta(n)=\text{max}\left\{g(K+\ww)\mid K+\ww\in\mathcal{K}_n^* \right\}.\]
\end{prop}
\begin{proof}
Let $K+\ww$ be a $\mathbb{Z}$-coset in $\mathcal{K}_n$ and let $\sigma : K+\ww \ra \Sigma_g$ be a representation of $\mathbb{Z}$-cosets, where $g=g(K+\ww)$. 
Note that one can write $\ww=\frac{d}{m}\kk$, where $d,m$ are relatively prime positive integers and $\kk$ is a primitive vector of $K$. 
Since $\sigma(K)\subseteq I_g$, there are integers $a_1,...,a_g$ such that $\sigma(\kk)=a_1\ee_1+\cdots+a_g\ee_g$. 
Moreover, we have
\[\sigma\left(\frac{d}{m}\kk\right)-\frac{1}{2}\vv[g] = \sum_{j=1}^g \left(\frac{d}{m}a_j-\frac{1}{2}\right) \ee_j \in I_g.\]
Therefore, we have $\frac{d}{m}a_j-\frac{1}{2}\in\mathbb{Z}$, that is, $da_j-\frac{m}{2}\in m\mathbb{Z}$ for any $1\le j \le g$. 
Hence, there is an positive integer $m_0$ such that $m=2m_0$ and we have
\[d \equiv 1 \text{ (mod }2\text{)} \quad \text{and} \quad a_j \equiv 0 \text{ (mod }m_0\text{)}\text{ for any }1\le j \le g.\]
Thus, $Q(\kk)\in m_0^2\mathbb{Z}$ and $B(\kk,K)\subseteq m_0\mathbb{Z}$.

Since $\kk$ is a primitive vector of $K$, we may assume that $K=\mathbb{Z}[\kk,\kk_2,...,\kk_n]$ and consider the $\mathbb{Z}$-lattice $\tilde{K}=\mathbb{Z}[\frac{\kk}{m_0},\kk_2,...,\kk_n]$ in the same quadratic space $\mathbb{Q}K$.
Note that $\frac{\kk}{m_0}$ is a primitive vector of $\tilde{K}$.
One may easily check that
\[\sigma(\tilde{K})\subseteq I_g  \quad \text{and} \quad \sigma\left(\frac{d}{2}\cdot\frac{\kk}{m_0}\right)-\frac{1}{2}\vv[g] \in I_g,\]
which implies that $\sigma : \tilde{K}+\frac{d}{2}\left(\frac{\kk}{m_0}\right) \ra \Sigma_g$ is a representation of $\mathbb{Z}$-cosets. 
Therefore, we have 
\[\tilde{K}+\frac{d}{2}\cdot\frac{\kk}{m_0} = \tilde{K}+\frac{1}{2}\cdot\frac{\kk}{m_0}\in \mathcal{K}_n^* \quad \text{and} \quad g\ge \tilde{g}:=g\left(\tilde{K}+\frac{1}{2}\cdot\frac{\kk}{m_0}\right).\]
On the other hand, if we let $\tilde{\sigma} : \tilde{K}+\frac{1}{2}\left(\frac{\kk}{m_0}\right) \ra \Sigma_{\tilde{g}}$ be a representation of $\mathbb{Z}$-cosets, then by restricting $\tilde{\sigma}$ on $K+\ww$ we obtain a representation of $K+\ww$ by $\Sigma_{\tilde{g}}$. 
Thus we may conclude $g\le \tilde{g}$. Hence $g=\tilde{g}$, which proves the proposition.
\end{proof}

\begin{rmk} \label{rmk1}
(a) Let $K+\frac{1}{2}\ww$ be a $\mathbb{Z}$-coset in $\mathcal{K}^*_n$, where $K=\mathbb{Z}[\dd_1,...,\dd_n]$ and let $M$ be the Gram matrix corresponding to $K$ with respect to the basis $\{\dd_1,...,\dd_n\}$.
Then we may assume that $\ww = \dd_{i_1}+\cdots +\dd_{i_t}$, where $1\le i_1 <\cdots < i_t\le n$ and $t\ge 1$.
Let $R$ be either $\mathbb{Z}$ or $\mathbb{Z}_2$ and let us consider $K+\frac{1}{2}\ww$ as an $R$-coset. 
Let $\sigma : K+\frac{1}{2}\ww \ra \Sigma_r$ be a representation of $R$-cosets, that is, $\sigma : K \ra I_r$ and $\sigma(\frac{1}{2}\ww)-\frac{1}{2}\vv[r]\in I_r$.
Let $T=(t_{ij})$ be the $n\times r$ matrix over $R$ such that $\sigma(\dd_i)=\sum_{j=1}^r t_{ij}\ee_j$ for any $1\le i \le n$.
Then the assumption that $\sigma$ is a representation of $R$-cosets is equivalent to the following conditions:
\begin{equation}\label{eq3}
M=TT^t \quad \text{and} \quad \sum_{i\in \{i_1,...,i_t\}} t_{ij} \equiv 1 \Mod{2} \text{ for each }1\le j \le r.
\end{equation}
Conversely, a matrix $T\in M_{n\times r}(R)$ satisfying \eqref{eq3} induces the representation of $R$-cosets $\sigma : K+\frac{1}{2}\ww \ra \Sigma_r$ defined by $\sigma(\dd_i)=\sum_{j=1}^r t_{ij}\ee_j$ for each $1\le i \le n$. 
Therefore, we shall identify the above $\sigma$ with $T$.\vspace{0.3cm}\\
(b) Let $M$ be an $n\times n$ symmetric matrix over $R$, which is not necessarily non-degenerate.
We will sometimes say $M+\frac{1}{2}\ww$ is represented by $\Sigma_r$, denoted by $M+\frac{1}{2}\ww \ra \Sigma_r$, which means that there exists an $n\times r$ integral matrix $T$ which satisfies \eqref{eq3}. 
Suppose that there are two symmetric matrices $M_1,M_2$ over $R$ such that 
\[
M_i+\frac{1}{2}\ww \ra \Sigma_{r_i} \quad \text{for each } i=1,2.
\]
If we let $T_i$ be the corresponding $n\times r_i$ integral matrix for each $i$, then the $n\times (r_1+r_2)$ matrix $T=(T_1 \ T_2)$ together with the $n\times n$ matrix $M=M_1+M_2$ satisfies (\ref{eq3}), hence we have 
$M+\frac{1}{2}\ww\ra \Sigma_{r_1+r_2}$.
\end{rmk}

We can simply analyze the problem in the case when $n=1$.

\begin{prop} \label{gdelta1}
We have $g_\Delta(1)=10$.
\end{prop}
\begin{proof}
Let $K+\frac{1}{2}\ww\in \mathcal{K}^*_n$.
As described in Remark \ref{rmk1} (a), we may assume that $\ww=\dd$, where $K=\mathbb{Z}[\dd]\cong \langle M \rangle$ for some positive integer $M$. 
Furthermore, finding a representation of $\mathbb{Z}$-cosets $\sigma : K+\frac{1}{2}\dd \ra \Sigma_r$ is equivalent to writing $M$ as a sum of $r$ squares of odd integers.\\
\indent We shall prove that every positive integer $M$ is a sum of at most $10$ squares of odd integers.
Clearly, $1$ and $2$ are sum of 1 and 2 odd squares, respectively. 
Now, let us assume that $M\equiv k \text{ (mod } 8\text{)}$ with $3\le k \le 10$.
Then, $M-(k-3)\equiv 3 \text{ (mod } 8\text{)}$ so that Legendre's three-square theorem implies that $M-(k-3)=t_1^2+t_2^2+t_3^2$ for some odd integers $t_1,t_2,t_3$. Since $k-3$ is a sum of $k-3$ squares of $1$, $M$ is a sum of $k$ odd squares. Thus $g_\Delta(1)\le 10$. On the other hand, every positive integers $M\equiv 2 \text{ (mod }8\text{)}$ which are not a sum of two squares, for example, the integer $42$, is a sum of $10$ odd squares. This proves the proposition.
\end{proof}

\section{Lemmas} \label{Lemmas}
In this section, we will introduce several lemmas.
We use the notations described in Remark \ref{rmk1} (a), so for a $\mathbb{Z}$-coset $K+\frac{1}{2}\ww \in \mathcal{K}_n^*$, we put $K=\mathbb{Z}[\dd_1,...,\dd_n]$ and $\ww = \dd_{i_1}+\cdots +\dd_{i_t}$, where $1\le i_1 <\cdots < i_t\le n$ and $t\ge 1$.
We begin with finding some necessary condition of a $\mathbb{Z}$-coset $K+\frac{1}{2}\ww \in \mathcal{K}_n^*$ to be represented by $\Sigma_r$.
\begin{lem}\label{lemneccond}
Let $K+\frac{1}{2}\ww$ be a $\mathbb{Z}$-coset in $\mathcal{K}_n^*$. If $K+\frac{1}{2}\ww$ is represented by $\Sigma_r$ for some positive integer $r$, then the following holds.
\begin{enumerate}[label=\upshape(\roman*), leftmargin=*, widest=iii]
	\item $Q(\kk) \equiv B(\ww,\kk) \Mod{2}$ for any $\kk \in K$.\label{neccond:1}
	\item $r\equiv Q(\ww) \Mod{8}$ and $r\le Q(\ww)$.\label{neccond:2}
	\item $r\le r_{K,\ww}$, where $r_{K,\ww}$ is the greatest positive integer satisfying 
	\[r_{K,\ww}\le Q(\dd_{i_1})+\cdots+Q(\dd_{i_t}) \text{ and } r_{K,\ww}\equiv Q(\ww) \Mod{8}.\]
\end{enumerate}
\end{lem}

\begin{proof}
Let $\sigma : K+\frac{1}{2}\ww \ra \Sigma_r$ be a representation of $\mathbb{Z}$-cosets and $T=(t_{ij})$ be the $n\times r$ integral matrix satisfying $\sigma(\dd_i)=\sum_{j=1}^r t_{ij}\ee_j$.
Then, 
\[
Q(\dd_i)=Q(\sigma(\dd_i))=\sum_{j=1}^r t_{ij}^2 \equiv \sum_{j=1}^r t_{ij} = B(\vv[r],\sigma(\dd_i))\equiv B(\ww,\dd_i) \Mod{2},
\]
for each $1\le i \le n$. Hence, for any $\kk=\sum_{i=1}^n k_i\dd_i \in K$, we have
\[Q(\kk) \equiv \sum_{i=1}^n k_i^2 Q(\dd_i) \equiv \sum_{i=1}^n k_i B(\ww,\dd_i)= B(\ww,\kk) \Mod{2}.\]

Now, we note that $\sigma(\ww)=\sum_{j=1}^r t_{\ww,j}\ee_j$, where $t_{\ww,j}=\sum_{i\in\{i_1,...,i_t\}}t_{ij}$ is an odd integer by the second condition of (\ref{eq3}), for each $1\le j \le r$.
Thus, we have
\[Q(\ww) = Q(\sigma(\ww)) = \sum_{j=1}^r t_{\ww,j}^2,\]
so that $Q(\ww)\equiv r \Mod{8}$ and $r\le Q(\ww)$. On the other hand, since not all elements in $\{t_{i_1j},...,t_{i_tj}\}$ are zero for each $1\le j \le r$, we have 
\[
r\le \sum_{j=1}^r  \left( \sum_{i\in\{i_1,...,i_t\}} t_{ij}^2 \right)=  
\sum_{i\in\{i_1,...,i_t\}}\left( \sum_{j=1}^r t_{ij}^2 \right)= Q(\dd_{i_1})+\cdots+Q(\dd_{i_t}).
\] Then, $r\le r_{K,\ww}$ follows from this with $r\equiv Q(\ww) \Mod{8}$.
\end{proof}

\begin{lem}\label{lem1}
Let $K$ be a $\mathbb{Z}_2$-lattice of rank $n$ and $\ww$ be a primitive vector of $K$. 
Suppose that $K+\frac{1}{2}\ww$ is represented by $\Sigma_{g'}$ over $\mathbb{Z}_2$ for some positive integer $g'$. Then we have $g' \equiv Q(\ww) \Mod{8}$.
Furthermore, we have
\[
K +\frac{1}{2}\ww\ra \Sigma_g
\]
for any positive integer $g\equiv Q(\ww) \Mod{8}$ satisfying
\[
g \ge 
\begin{cases} 
n+3 & \text{if }Q(\ww)\not\equiv 0 \Mod{4},\\ 
n+4 & \text{if }Q(\ww)\equiv 0 \Mod{4}. 
\end{cases}
\]
\end{lem}

\begin{proof}
For the sake of simplicity of notation, all the lattices, cosets, representations and matrices in the proof of this lemma are considered to be defined over $\mathbb{Z}_2$. 
Moreover, the $\mathbb{Z}_2$-lattice $I_r+\mathbb{Z}_2[\frac{1}{2}\vv[r]]$ associate with $\Sigma_r$ will be denoted by $L_r$ for any positive integer $r$ during the proof of this lemma.

Assume that there is a representation $\sigma' : K+\frac{1}{2}\ww \ra \Sigma_{g'}$ of $\mathbb{Z}_2$-cosets.
One may show that $Q(\ww) \equiv g' \Mod{8}$ holds by a similar argument used in the proof of \ref{neccond:2} of Lemma \ref{lemneccond}.

Now, we prove the second assertion.
Since $\sigma'(\frac{1}{2}\ww)\in L_{g'}$, the representation $\sigma'$ can be extended to a representation of $\mathbb{Z}_2$-lattices 
\[
\sigma' : K+\mathbb{Z}_2[\frac{1}{2}\ww] \ra L_{g'}.
\]
We shall divide the proof into three cases.

First, suppose that $Q(\ww)\equiv 1 \Mod{2}$. One may easily verify that 
\[
2\left( K + \mathbb{Z}_2[\frac{1}{2}\ww]\right)\cong \langle Q(\ww) \rangle \perp 2N,
\]
for some integral $\mathbb{Z}_2$-lattice $N$ and
\[
2L_{g} \cong 
\begin{cases}
\langle g \rangle \perp 4\left(\mathbb{H}\perp...\perp\mathbb{H} \right)& \text{if } g\equiv \pm 1 \Mod{8},\\
\langle g \rangle \perp 4\left(\mathbb{H}\perp...\perp\mathbb{H}\perp\mathbb{A} \right) & \text{if } g\equiv \pm3 \Mod{8},
\end{cases}
\]
where $\mathbb{H}\cong{\small	\begingroup
	\setlength\arraycolsep{3pt} \begin{pmatrix}0&1\\[-1pt]1&0 \end{pmatrix} \endgroup}$ and $\mathbb{A}\cong{\small 	\begingroup
	\setlength\arraycolsep{3pt} \begin{pmatrix}2&1\\[-1pt] 1&2 \end{pmatrix}\endgroup}$. 
It follows from Theorem 3 of \cite{OM1} that $N$ has no proper unimodular Jordan component since $2\left(K+\mathbb{Z}_2[\frac{1}{2}\ww]\right) \ra 2L_{g'}$.
On the other hand, the same theorem also implies that if $N$ has no proper unimodular Jordan component, then we have $2\left(K+\mathbb{Z}_2[\frac{1}{2}\ww]\right) \ra 2L_g$ for any integer $g\equiv Q(\ww) \Mod{8}$ with $g\ge n+3$. 
Therefore, there is a representation 
\[
\sigma : K+\mathbb{Z}_2[\frac{1}{2}\ww] \ra L_g.
\]
Let $\sigma(\frac{1}{2}\ww) = a_1\frac{\vv[g]}{2} + a_2 \ee_2 + \cdots + a_g\ee_g$, where the $a_i$'s are 2-adic integers.
Since 
\[
Q(\sigma(\ww)) = a_1^2g+4(a_2^2+\cdots+a_g^2+a_1(a_2+\cdots+a_g))\equiv a_1^2g \Mod{4},
\]
we have $g\equiv Q(\ww)=Q(\sigma(\ww)) \equiv a_1^2g \Mod{4}$.
Hence, we have
\[
a_1\equiv 1 \Mod{2} \quad \text{and}\quad \sigma(\ww)\equiv \vv[g] \Mod{2I_g}.
\]
Similarly, for any vector $\kk\in K$, let $\sigma(\kk) = b_1\frac{\vv[g]}{2} + b_2 \ee_2 + \cdots + b_g\ee_g$, where the $b_i$'s are 2-adic integers.
Since $Q(K)\subseteq \mathbb{Z}_2$ and 
\[Q(\kk)=Q(\sigma(\kk)) \equiv \frac{1}{4}b_1^2g\Mod{1},\] 
we have $b_1 \equiv 0 \Mod{2}$ so $\sigma(\kk)\in I_g$.
Hence, we have $\sigma(K)\subset I_g$ which implies that $\sigma : K+\frac{1}{2}\ww \ra \Sigma_g$ is a representation of $\mathbb{Z}_2$-cosets.

Next, suppose that $Q(\ww)\equiv 2 \Mod{4}$. 
One can verify that 
\[
2\left( K + \mathbb{Z}_2[\frac{1}{2}\ww]\right)\cong \langle Q(\ww) \rangle \perp 2N \text{  or  }
\langle Q(\ww) \rangle  \perp  \langle 2\varepsilon \rangle\perp 2N,\]
where $\varepsilon \in \mathbb{Z}_2^\times$, $N$ is an integral $\mathbb{Z}_2$-lattice and
\[
2L_g\cong 
\begin{cases}
\langle 2,2 \rangle \perp 4\left(\mathbb{H}\perp...\perp\mathbb{H} \right)& \text{if } g\equiv 2 \Mod{8},\\
\langle 6,14 \rangle \perp 4\left(\mathbb{H}\perp...\perp\mathbb{H}\perp\mathbb{A} \right) & \text{if } g\equiv 6 \Mod{8}.
\end{cases}
\]
In this case, $2\left(K+\mathbb{Z}_2[\frac{1}{2}\ww]\right) \ra 2L_{g'}$ implies that $N$ has no proper unimodular Jordan component and 
\[
\varepsilon\in
\begin{cases}1+4\mathbb{Z}_2 & \text{if } Q(\ww) \equiv 2 \Mod{8},\\
3+4\mathbb{Z}_2 & \text{if } Q(\ww) \equiv 6 \Mod{8}.
\end{cases}
\]
On the other hand, if the above conditions for both $N$ and $\varepsilon$ are satisfied, then we have $2\left(K+\mathbb{Z}_2[\frac{1}{2}\ww]\right) \ra 2L_g$ for any $g\equiv Q(\ww) \Mod{8}$ with $g\ge n+3$. 
Therefore, we have a representation of cosets $\sigma : K+\frac{1}{2}\ww \ra \Sigma_g$ by a similar reasoning to the case when $Q(\ww)\equiv 1 \Mod{2}$.

Finally, suppose that $Q(\ww) \equiv 0 \Mod{4}$. 
Since $\ww$ is a primitive vector of $K$, we may take $\{\ww,\dd_2,...,\dd_n\}$ as a basis for $K$. 
Let $T'=(t_{ij}')$ be the $n \times g'$ matrix over $\mathbb{Z}_2$ corresponding to the representation $\sigma'$, that is, 
\[(\sigma'(\ww),\sigma'(\dd_2),...,\sigma'(\dd_n) )^t= T'\cdot(\ee_1,...,\ee_{g'})^t.\]
Then $t_{1j}'\in 1+2\mathbb{Z}_2$ for any $1\le j \le g'$ (see Remark \ref{rmk1} (a)).
Now, we consider another $\mathbb{Z}_2$-lattice $\tilde{K}=\mathbb{Z}_2[\tilde{\mathbf{w}},\tilde{\mathbf{d}}_2,...,\tilde{\mathbf{d}}_n]$ whose Gram matrix with respect to the basis $\{\tilde{\mathbf{w}},\tilde{\mathbf{d}}_2,...,\tilde{\mathbf{d}}_n\}$ is $T'\cdot (T')^t +aE_{11}$, where $E_{11}$ is the $n\times n$ matrix with 1 in the $(1,1)$ position and $0$ elsewhere and an integer $a\in\{7,15\}$ is chosen to satisfy $d\tilde{K}_2\neq 0$. 
Then the $n\times (g'+a)$ matrix $\tilde{T}$ defined by
\[
\tilde{T} = \left(
\begin{array}{ccc|ccc}
&& & 1 & \cdots & 1               \\[-0em]
&&&0&\cdots&0                 \\ [-0.35em]
&T'&&&&                          \\  [-1.1em]
& && \vdots&\ddots&\vdots  \\ 
&&&0&\cdots&0 
\end{array}
\right)
\]
induces a representation of cosets $\tilde{\sigma} : \tilde{K}+\frac{1}{2}\tilde{\mathbf{w}} \ra \Sigma_{g'+a}$ (see Remark \ref{rmk1} (b)). 
Since $Q(\tilde{\mathbf{w}})\equiv 1\Mod{2}$, we may apply the result of the first case to conclude that there exists a representation of $\mathbb{Z}_2$-cosets $\sigma_0 : \tilde{K}+\frac{1}{2}\tilde{\mathbf{w}} \ra \Sigma_g$ for any integer $g\ge n+3$ satisfying $g\equiv Q(\ww)+7 \Mod{8}$. 
Let $T_0$ be the $n\times g$ integral matrix corresponding to $\sigma_0$, and define the $n\times (g+1)$ matrix $T$ as
\[
T = \left(
\begin{array}{ccc|c}
&& & \varepsilon_a              \\[-0em]
&&&0             \\ [-0.35em]
&T_0&                        \\  [-1.1em]
& && \vdots  \\ 
&&&0 
\end{array}
\right),
\]
where $\varepsilon_a \in \mathbb{Z}_2^\times$ such that $-a=\varepsilon_a^2$. Then $T$ induces a representation of $\mathbb{Z}_2$-cosets
\[\sigma : K+\frac{1}{2}\ww \ra \Sigma_{g+1}.\] 
This proves the lemma since $g+1\ge n+4$ and $g+1\equiv Q(\ww) \text{ (mod }8\text{)}$.
\end{proof}

\begin{lem}\label{lem2}
{\rm (1)} Let $K=\mathbb{Z}[\ww,\ww']$ be a positive definite integral $\mathbb{Z}$-lattice such that $Q(\ww)\equiv 5 \Mod{8}$ and $B(\ww,\ww') \equiv Q(\ww') \Mod{2}$.
Then the $\mathbb{Z}$-coset $K+\frac{1}{2}\ww$ is represented by $\Sigma_5$.

\noindent {\rm (2)} Let $K=\mathbb{Z}[\ww]\perp K'$ be a $\mathbb{Z}$-lattice such that $Q(\ww)=6$ and $K'$ is a positive definite even integral $\mathbb{Z}$-lattice of rank $2$.
Then $K+\frac{1}{2}\ww$ is represented by $\Sigma_6$.
\end{lem}
\begin{proof}
Since the proofs for (1) and (2) are quite similar to each other, we only provide the proof of (1).
For the sake of convenience, put $a=Q(\ww)$, $b=B(\ww,\ww')$ and $c=Q(\ww')$.
By Corollary \ref{localglobal}, it is enough to show, for any prime $p$, that 
\[
K_p+\frac{1}{2}\ww \ra (\Sigma_5)_p.
\]
In case when $p\neq 2$,
$K_p+\frac{1}{2}\ww = K_p$ is represented by $(I_5)_p=(\Sigma_5)_p$ by Theorem 2 of \cite{OM1}. When $p=2$, we have, by hypothesis, that
\[
2(K_2+\mathbb{Z}_2[\frac{1}{2}\ww])\cong \langle a \rangle\perp\langle 4a(ac-b^2) \rangle \quad \text{and} \quad 4a(ac-b^2)\in 8\mathbb{Z}_2.
\]
Therefore, by following the argument of the first case of the proof of Lemma \ref{lem1} similarly, one may conclude $2(K_2+\mathbb{Z}_2[\frac{1}{2}\ww])$ is represented by $2((I_5)_2+\mathbb{Z}_2[\frac{1}{2}\vv[5]])$ so that $K_2+\frac{1}{2}\ww$ is represented by $(\Sigma_5)_2$.
\end{proof}

Let $n$ be a positive integer and let $i,j$ be integers such that $1\le i, j \le n$, let $E_{ij}$ be the $n\times n$ matrix with 1 in the $(i,j)$ position and $0$ elsewhere. 

\begin{lem}\label{matrixsplit}
Let $n\ge3$ be a positive integer and let $n_0$ be an integer such that $1\le n_0 \le n$. Let $A= {\rm diag}(a_1,...,a_n)$ be a diagonal matrix in $M_n(\z)$ and $S=(s_{ij})$ be a symmetric matrix in $M_n(\z)$.
Suppose that $A,S$ and $n_0$ satisfy the following conditions:
\begin{enumerate}[label=\upshape(\roman*), leftmargin=*, widest=iii]
	\item $a_i+s_{ii} \equiv s_{in_0} \Mod{2}$ for any $i\neq n_0$,\label{splitcond:1}
	\item $a_i>2n(n-1)(3n+2)$ for any $i$,\label{splitcond:2}
	\item $a_ia_j\ge4n^2|s_{ij}|^2$ for any $1\le i \le j \le n$.\label{splitcond:3}
\end{enumerate}
Then, $A+S$ is a positive definite symmetric matrix and we have
\[
K+\frac{1}{2}\ww \ra \Sigma_{6\cdot\frac{(n-1)(n-2)}{2}+5(n-1)+k_0}
\] 
for some integer $0\le k_0 \le 10$, where  $K=\z[\dd_1,...,\dd_n]$ whose Gram matrix with respect to $\{\dd_1,...,\dd_n\}$ is $A+S$ and $\ww=\dd_{n_0}$.
\end{lem}

\begin{proof}
By condition \ref{splitcond:1}, we can write, for each $i\neq n_0$, $a_i+s_{ii}=\sum_{1\le j \le n, \,j\neq i} t_{ij}$ such that 
\[
t_{ij}\equiv 
\begin{cases} 0 \Mod{2} & \text{  if } j\neq n_0,\\
s_{in_0} \Mod{2} & \text{  if } j=n_0 ,
\end{cases}\quad \text{and} \quad t_{ij}\ge 2 \left\lfloor \frac{a_i+s_{ii}-1}{2(n-1)}\right\rfloor.
\]
Since $a_i\ge 2n|s_{ii}|$ by condition \ref{splitcond:2}, we have 
\[
a_i+s_{ii}\ge\frac{2n-1}{2n}a_i = \frac{n-1}{n}a_i+\frac{1}{2n}a_i.
\]
Hence, by condition \ref{splitcond:3}, one may verify that
\[
t_{ij}\ge2 \left\lfloor \frac{a_i+s_{ii}-1}{2(n-1)}\right\rfloor > \frac{a_i+s_{ii}-1}{n-1}-2 \ge \frac{a_i}{n}+\frac{a_i-2n(2n-1)}{2n(n-1)}>\frac{a_i}{n}.
\]
Similarly, we can write $a_{n_0}+s_{n_0n_0}=6\cdot\frac{(n-1)(n-2)}{2}+\sum_{1\le j \le n, \,j\neq n_0} t_{n_0j}+r_0$, where
\[
t_{n_0j}\equiv 5 \text{ (mod }8\text{)},\,\, t_{n_0j}\ge \left\lfloor \frac{a_{n_0}+s_{n_0n_0}}{n-1}-3(n-2)\right\rfloor-7,
\]
and $0\le r_0\le 7(n-1)$.
One may also show that $t_{n_0j}>\frac{a_{n_0}}{n}$. Therefore, by condition \ref{splitcond:3}, we have $t_{ij}t_{ji}>\frac{a_ia_j}{n^2}\ge 4|s_{ij}|^2\ge |s_{ij}|^2$ for any $i\neq j$.

Now, we can decompose $A+S$ as follows:
\[
\begin{array}{rcl}
A+S &=& \displaystyle\sum_i (a_i+s_{ii})E_{ii} + \sum_{1\le i,j \le n} s_{ij}E_{ij}\\
&= & \displaystyle\sum_{ \substack{i<j \\ i,j\neq n_0}}
(6E_{n_0n_0}+t_{ij}E_{ii}+t_{ji}E_{jj}+s_{ij}E_{ij}+s_{ji}E_{ji})\\
&& + \displaystyle\sum_{j\neq n_0}(t_{n_0j}E_{n_0n_0}+t_{jn_0}E_{jj}+s_{n_0j}E_{n_0j}+s_{jn_0}E_{jn_0}) + r_0E_{n_0n_0}.
\end{array}
\]
Hence, one may easily observe that $A+S$ is positive definite. Moreover, since $t_{ij}t_{ji}-|s_{ij}|^2>0$ for any $i\neq j$, we can apply Lemma \ref{lem2} so that each 
\[
(6E_{n_0n_0}+t_{ij}E_{ii}+t_{ji}E_{jj}+s_{ij}E_{ij}+s_{ji}E_{ji})+\frac{1}{2}\ww
\]
is represented by $\Sigma_6$ for any $i<j$ with $i,j\neq i_t$, and for each $j\neq n_0$,
\[(t_{n_0j}E_{n_0n_0}+t_{jn_0}E_{jj}+s_{n_0j}E_{n_0j}+s_{jn_0}E_{jn_0})+\frac{1}{2}\ww\]
is represented by $\Sigma_5$. 
Furthermore, $r_0E_{n_0n_0} + \frac{1}{2}\ww$ can be represented by $\Sigma_{k_0}$ for some $0\le k_0 \le 10$, for $g_\Delta(1)=10$.
Thus, 
\[(A+S) + \frac{1}{2}\ww \ra \Sigma_{6\cdot\frac{(n-1)(n-2)}{2}+5(n-1)+k_0}\]
for some $0\le k_0 \le 10$, which proves the lemma (see Remark \ref{rmk1} (b)).
\end{proof}
\vspace{0.1cm}

\section{Upper bound for $g_\Delta(n)$}\label{ubdsec}

In this section, we will derive an upper bound for $g_\Delta(n)$ and complete the proof of Theorem \ref{thm1}.
We begin by describing the ``balanced HKZ reduction'' introduced in Section 4 of \cite{BCIL} in terms of $\mathbb{Z}$-lattices. Let $U(n)$ be the group of upper triangular unipotent matrices in $M_n(\mathbb{R})$.
Let $K$ be a positive definite $\mathbb{Z}$-lattices of rank $n$ and let $\{\dd_1,...,\dd_n\}$ be a basis for $K$.
We say that a basis $\{\dd_1,...,\dd_n\}$ for $K$ is balanced HKZ-reduced if its corresponding Gram matrix $M$ is of the form $H[X]:=X^tHX$, where $X=(x_{ij})\in U(n)$ and $H=\text{diag}(h_1,...,h_n)$ satisfy the following two properties:\\

\noindent(1) $h_1=\mu(K)$ and $h_ih_j^{-1}\le \alpha(j-i)$ for any $1 \le i < j \le n$;\\
\noindent(2) $|x_{ij}|,|y_{ij}|\le c(j-i)$ for any $1 \le i \le j \le n$, where $X^{-1}=(y_{ij})$.\\

\noindent Here, $\alpha(m):=\sigma_{m+1}\prod_{k=2}^{m+1}\sigma_k^{\frac{1}{k-1}}$, $\sigma_k=4\pi^{-1}\Gamma\left(\frac{k}{2}+1\right)^{\frac{2}{k}}$ and $c(m)$ is the coefficient of $x^m$ in the Maclaurin series of $e^{\frac{x/2}{1-x}}$. Note that every positive definite $\mathbb{Z}$-lattice has a ``balanced HKZ-reduced'' basis (see \cite[Section 4]{BCIL}). On the other hand, we can bound the values $\alpha(j-i)$ (\cite[Corollary 2.5]{S}) as
\begin{equation}\label{bdh}
\alpha(j-i) \le \overline\alpha(n):=e^{\text{ln}(n+1)+(\text{ln}(n+1))^2}.
\end{equation}
Furthermore, there exists an absolute constant $D\ge 1$ such that 
\vspace{-0.1em}
\begin{equation}\label{bdx}
2c(j-i) \le \overline{c}(j-i):=De^{\sqrt{2(j-i)}},
\end{equation}
for any $1 \le i \le j \le n$  (\cite{P}, see also \cite[p.547]{Kn}). Note that $\overline{c}(m)$ is an increasing function of $m$.

\begin{prop}\label{mainprop}
Let  $n\ge3$ be an integer and let  
\[G(n):=144D^6 n^{12} e^{4(\text{ln}(n+1)+(\text{ln}(n+1))^2)} e^{(4+4\sqrt{2})\sqrt{n/2}},\]
where $D$ is the absolute constant in \eqref{bdx}. 
Then every $\z$-coset $K+\frac{1}{2}\ww$ satisfying condition \ref{neccond:1} of Lemma \ref{lemneccond} can be represented by $\Sigma_{6\cdot\frac{(n-1)(n-2)}{2}+5(n-1)+n+k_0}$ for some integer $0\le k_0 \le 10$, provided that $\mu(K)\ge G(n)$.
\end{prop}
\begin{proof}
The proof of this proposition is motivated by Section 6 of \cite{BCIL} and a modification of the arguments in there. 
The strategy of the proof is outlined as follows. We will take a specific basis for $K$ whose Gram matrix will be denoted by $M$.
Then we will take a diagonal matrix $A=\text{diag}(a_1,...,a_n)$, with all the $a_i$'s as large as possible, such that $M-A$ remains positive semidefinite. 
Then we will take $P\in M_n(\z)$ such that $P^tP$ approximates $M-A$ well and $P^tP+\frac{1}{2}\ww$ is represented by $\Sigma_n$. 
Write $M-A$ as $P^tP+S$, or equivalently, $M=P^tP+A+S$. We will show that $A$ and $S$ satisfy all conditions in Lemma \ref{matrixsplit}. 
As a result, $(A+S) +\frac{1}{2}\ww$ will be represented by $\Sigma_{6\cdot\frac{(n-1)(n-2)}{2}+5(n-1)+k_0}$ for some $0\le k_0 \le 10$. 
Hence we will conclude that $M+\frac{1}{2}\ww$ can be represented by $\Sigma_{6\cdot\frac{(n-1)(n-2)}{2}+5(n-1)+n+k_0}$ for some $1\le k_0\le 10$.\\

Let $\{\dd_1,...,\dd_n\}$ be a balanced HKZ-reduced basis for $K$ whose corresponding Gram matrix is $H[X]$, where $H$ is a diagonal matrix $\text{diag}(h_1,...,h_n)$,  $h_1=\mu(K)$, and $X\in U(n)$ which satisfy (1) and (2).
Let $\ww = \dd_{i_1}+\cdots +\dd_{i_t}$ for some $1\le i_1 <\cdots < i_t\le n$ and $t\ge 1$. 
With respect to the basis obtained by replacing $\dd_{i_t}$ with $\ww$, $K\cong M:=H[X']$ where $X':=XT$ and $T\in U(n)$ is defined as $T:=I_n + E_{i_1i_t} + \cdots + E_{i_{t-1}i_t}$.
We note that $T^{-1}=I_n - (E_{i_1i_t} + \cdots + E_{i_{t-1} i_t})$. 
If we put $X'=(x_{ij}')$ and $(X')^{-1}=T^{-1}X^{-1}=(y_{ij}')$, then by a straight forward computation using (\ref{bdx}) we obtain
\begin{equation}\label{bdxy}
|x_{ij}'| \le n\bar{c}(j-i), \quad |y_{ij}'| \le \overline{c}(j-i) \quad \text{ for any } 1\le i<j \le n.
\end{equation}

Now, for any $1\le k \le n$, we let 
\[
a_k:= \left\lfloor \frac{1}{n^2}\overline{\alpha}(n)^{-1} \overline{c}(n-k)^{-2}h_k \right\rfloor,
\]
and let $A:=\text{diag}(a_1,...,a_n)$.
Following the same argument used in the proof of Proposition 6.3 of \cite{BCIL}, we can find an upper triangular matrix $N=(n_{ij})$ such that 
\[
I_n-A[X'^{-1}\sqrt{H}^{-1}]=N^tN.
\]
Note that $|n_{ij}|\le 1$ for any $i\le j$, since $1-\sum_{i\le j}|n_{ij}|^2$ is the $(j,j)$ entry of $I_n-N^tN=A[X'^{-1}\sqrt{H}^{-1}]$ which is positive semidefinite.

Let $W=(w_{ij})$ be the upper triangular matrix $N\sqrt{H}X'$ in $M_n(\mathbb{R})$. Then $W^tW=M-A$. 
We can take an integral matrix $P=(p_{ij})$ satisfying
\[
|w_{ij}-p_{ij}|\le 1 \quad \text{and} \quad p_{ii_t}\equiv 1 \Mod{2}
\]
for any $1\le i,j\le n$. 
Let $Q=(q_{ij}):=W-P$. 
Then
\[
P^tP = (W-Q)^t(W-Q) = M-A-Q^tW-W^tQ+Q^tQ,
\]
hence $M=P^tP+A+S$, where $S=(s_{ij}):=Q^tW+W^tQ-Q^tQ$ which is an integral symmetric matrix.
We note that $P^tP+\frac{1}{2}\ww$ is represented by $\Sigma_n$. Therefore, as outlined at the beginning of the proof, it is enough to show that $A$ and $S$ satisfy the conditions in Lemma \ref{matrixsplit} with $n_0=i_t$.

To verify the first condition, let $M=(m_{ij})$ and note that $m_{ii} \equiv m_{ii_t} \Mod{2}$ for any $i$, by the hypothesis of this proposition. 
Also, the $(i,i)$ and the $(i,i_t)$ entries of $P^tP$ have the same parity for any $i$, by the construction of $P$. Since $A+S=M-P^tP$, the first condition in Lemma \ref{matrixsplit} is satisfied.

Now we estimate the lower bound of $a_i$.
By the hypothesis, we have
\[
\overline{\alpha}(n)h_i\ge h_1 = \mu(K) \ge G(n)= 144D^6 n^{12} e^{4(\text{ln}(n+1)+(\text{ln}(n+1))^2)} e^{(4+4\sqrt{2})\sqrt{n/2}}.
\]
Combining this with the fact that $\overline{c}(n-j)^2\overline{c}(j)^2 = D^4e^{2\sqrt{2}(\sqrt{n-j}+\sqrt{j})}$ is maximized at $j=\frac{n}{2}$, we have 
\begin{equation}\label{lbdalphah}
\overline{\alpha}(n)h_i \ge  144n^{12} \overline{\alpha}(n)^4\overline{c}(n)^2\overline{c}(n-j)^2\overline{c}(j)^2,
\end{equation}
for any $1\le j\le n$. Hence, we have $\frac{1}{n^2}\overline{\alpha}(n)^{-1} \overline{c}(n-i)^{-2}h_i\ge n^{10} \ge1$, so that
\begin{equation}\label{bda}
a_i= \left\lfloor \frac{1}{n^2}\overline{\alpha}(n)^{-1} \overline{c}(n-i)^{-2}h_i \right\rfloor \ge \frac{1}{2n^2}\overline{\alpha}(n)^{-1} \overline{c}(n-i)^{-2}h_i,
\end{equation}
and, especially $a_i\ge n^{10}>2n(n-1)(3n+2)$ for any $i$. This proves that the second condition in Lemma \ref{matrixsplit} is satisfied.

On the other hand, using \eqref{bdh}, \eqref{bdxy}, and the fact that $|n_{ij}|\le1$, one may obtain that
$|w_{ij}|\le n^2 \overline{c}(j)(\overline{\alpha}(n)h_j)^\frac{1}{2}$ for any $1\le i \le j \le n$. Furthermore, since $S=Q^tW+W^tQ-Q^tQ$ and $|q_{ij}|\le 1$, one may show for each $1\le i \le j \le n$ that
\begin{equation}\label{bds}
|s_{ij}|\le 2n^3 \overline{c}(j)(\overline{\alpha}(n)h_j)^\frac{1}{2}+n \le 3n^3\overline{c}(j)(\overline{\alpha}(n)h_j)^\frac{1}{2}.
\end{equation}
Thus, by \eqref{lbdalphah}, \eqref{bda}, and \eqref{bds}, for any $1\le i \le j \le n$, we have
\[
\begin{array}{rcl}
\dfrac{a_ia_j}{n^2}&\ge&\dfrac{1}{4n^6} \overline{\alpha}(n)^{-2}\overline{c}(n-i)^{-2}\overline{c}(n-j)^{-2}h_ih_j \\
&\ge&4|s_{ij}|^2\dfrac{\overline\alpha(n)h_i}{144n^{12}\overline{c}(j)^2\overline{c}(n-j)^2\overline{c}(n-i)^2\overline{\alpha}(n)^4}\\
&\ge&4|s_{ij}|^2.
\end{array}
\]
This implies that the third condition in Lemma \ref{matrixsplit} is satisfied, hence we complete the proof.
\end{proof}

\begin{prop}\label{proprecgn}
For any positive integer $n\ge3$,
\[g_\Delta(n)\le \text{max} \left\{ g_\Delta(n-1)+G(n), \, 3n^2-3n+11 \right\},\]
where $G(n)$ is the function defined in Proposition \ref{mainprop}.
\end{prop}
\begin{proof}
Let $K+\frac{1}{2}\ww$ be a $\mathbb{Z}$-coset in $\mathcal{K}_n$. If $\mu(K)\ge G(n)$, then by Proposition \ref{mainprop}, $K+\frac{1}{2}\ww$ is represented by $\Sigma_g$ for some integer less than or equal to 
\[6\cdot\frac{(n-1)(n-2)}{2}+5(n-1)+n+10=3n^2-3n+11.\]
\indent Suppose that $\mu(K)<G(n)$. 
We may assume that $K+\frac{1}{2}\ww$ is represented by $\Sigma_r$ for some $r$. 
Furthermore we may also assume that $r\ge G(n)$.
Let $K=\mathbb{Z}[\dd_1,...,\dd_n]$, $\frac{1}{2}\ww=w_1\dd_1+\cdots+w_n\dd_n$ and define 
\[f(x_1,...,x_n):=4\cdot Q(x_1\dd_1+\cdots+x_n\dd_n+\frac{1}{2}\ww).\]
We may further assume that $\{\dd_1,...,\dd_n\}$ is a balanced HKZ reduced basis for $K$ so that $Q(\dd_1)=\mu(K)$. As is described in \eqref{eq2}, there are $r$ linear forms $L_1(x_1,...,x_n),...,L_r(x_1,...,x_n)$ over $\mathbb{Z}$ and integers $c_1,...,c_r$ such that 
\[
f(x_1,...,x_n) = \sum_{j=1}^r 4\cdot\left(L_j(x_1,...,x_n)+c_j+\frac{1}{2}\right)^2.
\]
Since $f(-w_1,...,-w_n)=0$, we have $-L_j(w_1,...,w_n)+c_j+\frac{1}{2}=0$ for any $1\le j \le r$.
Hence, for any $1\le j \le r$, we have
\[L_j(x_1,...,x_n)+c_j+\frac{1}{2} = L(x_1+w_1,...,x_n+w_n).\]
If $b_1,...,b_r$ are the coefficients of $x_1$ in $L_1,...,L_r$ respectively, then at most $\lfloor G(n) \rfloor$ of them are nonzero. Thus, without loss of generality, we can write
\[
\begin{array}{rcl}
f(x_1,...,x_n) &= &\displaystyle\sum_{j=1}^{\lfloor G(n)\rfloor} 4\cdot\left(L_j(x_1,...,x_n)+c_j+\frac{1}{2}\right)^2\\
&&+\displaystyle\sum_{j=\lfloor G(n)\rfloor+1}^r 4\cdot\left(L_j(0,x_2,...,x_n)+c_j+\frac{1}{2}\right)^2.
\end{array}
\]
Note that $(L_j(0,x_2,...,x_n)+c_j+\frac{1}{2})^2=L_j(0,x_2+w_2,...,x_n+w_n)^2$. Thus, the second sum is zero or a complete quadratic polynomial in $n-1$ variables represented by $\Delta_{r-\lfloor G(n)\rfloor}$.
Hence it is represented by $\Delta_g$ for some integer $g\le g_\Delta(n-1)$. Hence, the proposition follows immediately from this.
\end{proof}

\begin{proof}[Proof of Theorem \ref{thm1}]
\noindent {\em } Clearly, $G(n)> 3n^2-3n+11$. Hence, by Proposition \ref{proprecgn}, 
\[g_\Delta(n)\le \sum_{j=3}^n G(j) + g_\Delta(2), \quad \text{for }n\ge3.\]
We will show that $g_\Delta(2)=12$ in Section \ref{gdelta1234}. Therefore, we have
\[g_\Delta(n)\le nG(n)=144D^6 n^{13} e^{4(\text{ln}(n+1)+(\text{ln}(n+1))^2)} e^{(4+4\sqrt{2})\sqrt{n/2}}.\]

\noindent Since $144D^6n^{13}e^{4(\text{ln}(n+1)+(\text{ln}(n+1))^2)}=O\left(e^{\varepsilon\sqrt{n}}\right)$ for any $\varepsilon>0$, we may conclude that 
\[ g_\Delta(n) = O\left(e^{(4+2\sqrt{2}+\varepsilon)\sqrt{n}}\right). \]
\end{proof}

\section{Exact value of $g_\Delta(n)$ for $2\le n \le 4$}\label{gdelta1234}
In this section, we always assume that $n$ is an integer such that $2\le n \le 4$ and we will determine the exact value of $g_\Delta(n)$. 
Let $K+\frac{1}{2}\ww$ be a $\mathbb{Z}$-coset in $\mathcal{K}_n^*$. 
From now on, we fix the following notations. 
We write $K=\mathbb{Z}[\dd_1,...,\dd_n]$ and $\ww = \dd_{i_1}+\cdots +\dd_{i_t}$ for some $t\ge1$ and  $1\le i_1 <\cdots < i_t\le n$. 
We denote the corresponding Gram matrix of $K$ with respect to $\{\dd_1,...,\dd_n\}$ by $M=(m_{ij})$ and we assume that $M$ is a Minkowski reduced symmetric matrix. 
By \cite{C} (see Lemma 1.2  of page 257), we have 
\begin{equation}\label{cond1}
0<m_{11}\le m_{22}\le ... \le m_{nn} \,\,\text{ and }  \,\, |2m_{ij}|\le m_{ii} \,\,\,\forall \,1\le i<j\le  n.
\end{equation}
We shall state two more technical lemmas, which will be used in the proof of Theorem \ref{thm2}.

\begin{lem}\label{lem3}
Let $Q(\mathbf{x})=Q(x_1\dd_1+\cdots+x_n\dd_n)$ be a positive definite quadratic form whose Gram matrix is a Minkowski reduced symmetric matrix $M=(m_{ij})$. Then, for any $1\le i \le n$, we have
\[Q(\mathbf{x})\ge C(n)m_{ii}x_i^2,\]
where $C(2)=\frac{3}{4}$, $C(3)=\frac{1}{2}$ and $C(4)=\frac{1}{5}$.
\end{lem}
\begin{proof}
We only provide a proof in the case when $n=4$. Other cases can be proved similarly ({\em cf.} see Lemma 2.3 of \cite{CO}). 
Fix an integer $i$ in $\{1,2,3,4\}$ and let $j<k<l$ be the remaining three integers listed in increasing order. Let
\[D_{jkl}=m_{jj}m_{kk}m_{ll}-m_{jj}m_{kl}^2-m_{kk}m_{jk}^2-m_{ll}m_{jk}^2+2m_{jk}m_{kl}m_{jl},\]
which is the determinant a $3\times 3$ submatrix of $M$. Hence, $D_{jkl}$ is positive, since $Q(\mathbf{x})$ is positive definite. 
From the fact that $\gamma_4^4=4$, where $\gamma_4$ is the $4$-dimensional Hermite constant, we have $m_{11}m_{22}m_{33}m_{44}\le 4D$, where $D$ is the discriminant of $Q$. We refer readers to \cite[Theorem 2.2, 3.1 of Chapter 12]{C} and \cite[Satz 7]{v} for more details.
By \eqref{cond1}, we have
\[D_{jkl}m_{ii}\le \left( \frac{5}{4}m_{jj}m_{kk}m_{ll} \right)m_{ii} \le 5D.\]
Now, by completing the squares, we have
\[Q(x_1\dd_1+\cdots+x_4\dd_4)\ge m_{jj}(x_j+\cdots)^2+\cdots+\frac{D}{D_{jkl}}x_i^2 \ge \frac{1}{5}m_{ii}x_i^2.\]
Hence, we prove the lemma.
\end{proof}

\begin{lem} \label{lem4}
Let $K+\frac{1}{2}\ww\in \mathcal{K}_n^*$ and $Q(\ww)\equiv k \pmod{8}$ with $1\le k \le 8$. Furthermore, let $S_2=\{5,6,7,8\}$, $S_3=\{6,7,8\}$ and $S_4=\{7,8\}$.\vspace{0.3em}\\
{\rm (1)} If $Q(\ww)\le 8$ or $k\in S_n$, then $K+\frac{1}{2}\ww \ra \Sigma_k$.\vspace{0.3em}\\
{\rm (2)} Suppose that $Q(\ww)>8$ and $k\not\in S_n$. If there are non-negative integers $k_{i_1},...,k_{i_t}$ such that 
\[ k_{i_1}+\cdots+k_{i_t}=k \quad \text{and} \quad Q(x_1\dd_1+\cdots+x_n\dd_n)-(k_{i_1}x_{i_1}^2+\cdots+k_{i_t}x_{i_t}^2)\]
is a positive definite quadratic form. Then $K+\frac{1}{2}\ww \ra \Sigma_{k+8}$.
\end{lem}
\begin{proof}
(1) If $Q(\ww)\le8$ then the result follows from Lemma \ref{lemneccond}. 
Now we assume that $k\in S_n$. 
Since $k\ge n+3$, $K_p+\frac{1}{2}\ww=K_p$ is represented by $(I_k)_p=(\Sigma_k)_p$ for any prime $p\neq 2$. 
Also, by Lemma \ref{lem1}, $K_2+\frac{1}{2}\ww$ is represented by $(\Sigma_k)_2$. 
Thus, by Corollary \ref{localglobal}, $K+\frac{1}{2}\ww \ra \Sigma_k$.

\noindent (2) Let $\sigma_0 : K+\frac{1}{2}\ww \ra \Sigma_r$ be a representation of $\mathbb{Z}$-cosets.
Consider another $\mathbb{Z}$-coset $K'+\frac{1}{2}\ww'$, where $K'=\mathbb{Z}[\dd_1',...,\dd_n']$ is a $\mathbb{Z}$-lattice whose Gram matrix with respect to $\{\dd_1',...,\dd_n'\}$ is equal to 
\[M'=M-(k_{i_1}E_{i_1i_1}+\cdots+k_{i_t}E_{i_ti_t})\] 
and $\ww'=\dd_{i_1}'+\cdots+\dd_{i_t}'$. 
We note that $K'$ is positive definite and $Q(\ww') = Q(\ww)-k$ is a positive integer congruent to $0$ modulo $8$ by the hypothesis.

Let $T_0$ be the $n\times r$ integral matrix corresponding to $\sigma_0$ and let $\varepsilon$ be a unit in $\mathbb{Z}_2$ such that $-1=7\varepsilon^2$. 
We consider the following $n\times (r+7k)$ matrix $T'$ over $\mathbb{Z}_2$:
\vspace{0.5em}
\[
\begin{array}{rrl}
T':=&\left(
\begin{array}{ccc|ccc}
&& & \varepsilon \,\cdot\,\cdot\,\cdot\, \varepsilon&  & \\[-0.7em]
&T_0 && &\ddots&  \\ [-0.2em]
&&&&&\varepsilon \,\cdot\,\cdot\,\cdot\, \varepsilon
\end{array}
\right)
&
\begin{array}{l}
\leftarrow \quad i_1\text{-th row}\\
\\
\leftarrow \quad i_t\text{-th row.}
\end{array}
\\[-0.6em]
&\underbrace{\qquad\quad\,\,}_{7k_{i_1}\text{-copies}}
\hspace{0.92cm}
\underbrace{\qquad\quad\,\,}_{7k_{i_t}\text{-copies}}
\hspace{0.37cm}
\\
\end{array}
\]
Here, $\varepsilon$'s are all placed on $i_a$-th row for each $1\le a \le t$, only one $\varepsilon$ is placed on each column and 0's are placed elsewhere.
Then $T'$ induces a representation $\sigma' : (K')_2+\frac{1}{2}\ww' \ra (\Sigma_{r+7k})_2$ (see Remark \ref{rmk1} (a)), hence by Lemma \ref{lem1}, we have
\[(K')_2+\frac{1}{2}\ww' \ra (\Sigma_8)_2.\]
It is clear that $K'_p+\frac{1}{2}\ww=K'_p$ is represented by $(I_8)_p=(\Sigma_8)_p$ for any prime $p\neq 2$. 
Thus, by Corollary \ref{localglobal}, there is a representation of $\mathbb{Z}$-cosets $\sigma_1 : K'+\frac{1}{2}\ww' \ra \Sigma_8$. 
If we let $T_1$ be the $n\times 8$ matrix corresponding to $\sigma_1$, then the following $n\times (k+8)$ matrix $T$ over $\mathbb{Z}$
\[
\begin{array}{rrl}
T:=&\left(
\begin{array}{ccc|ccc}
&& & 1 \,\cdot\,\cdot\,\cdot\, 1&  & \\[-0.7em]
&T_1 && &\ddots&  \\ [-0.2em]
&&&&&1 \,\cdot\,\cdot\,\cdot\, 1
\end{array}
\right)
&
\begin{array}{l}
\leftarrow \quad i_1\text{-th row}\\
\\
\leftarrow \quad i_t\text{-th row,}
\end{array}
\\[-0.6em]
&\underbrace{\qquad\quad\,\,}_{k_{i_1}\text{-copies}}
\hspace{1.10cm}
\underbrace{\qquad\quad\,\,}_{k_{i_t}\text{-copies}}
\hspace{0.42cm}
\\
\end{array}
\]
induces a representation of $\mathbb{Z}$-cosets $\sigma : K+\frac{1}{2}\ww \ra \Sigma_{k+8}$.
\end{proof}

We are now ready to prove Theorem \ref{thm2}. First, we shall prove the following Proposition.

\begin{prop}\label{lbdgdeltan}
We have $g_\Delta(n)\ge n+10$ for any $2\le n \le 4$.
\end{prop}
\begin{proof}
 Let $K=\mathbb{Z}[\dd_1,\dd_2]$ be a $\mathbb{Z}$-lattice whose Gram matrix with respect to $\{\dd_1,\dd_2\}$ is ${\small\begin{pmatrix}8&2\\ 2&12 \end{pmatrix}}$ and $\ww=\dd_2$. 
If $K+\frac{1}{2}\ww$ is represented by $\Sigma_r$, then $r\equiv 4 \Mod{8}$ by Lemma \ref{lemneccond}. Note that the following matrix
\[T:=\begin{pmatrix} 
1\hspace{0.8em}1\hspace{0.8em}1&1&1&-1&-1&-1&0&0&0&0 \\ 
1\hspace{0.8em}1\hspace{0.8em}1&1&1&1&1&1&1&1&1&1 \end{pmatrix}\]
induces a representation of $\mathbb{Z}$-cosets $K+\frac{1}{2}\ww\ra \Sigma_{12}$ (see Remark \ref{rmk1} (a)). However, $K+\frac{1}{2}\ww$ cannot be represented by $\Sigma_{4}$, since $K$ cannot be represented by $I_4$ over $\mathbb{Q}_2$. Thus, we have $g_\Delta(2)\ge12$.\\
\indent For the case when $n=3$ or $4$, we consider a $\mathbb{Z}$-lattice $K=\mathbb{Z}[\dd_1,...,\dd_n]$ whose Gram matrix with respect to $\{\dd_1,...,\dd_n\}$ is a diagonal matrix 
\[
\text{diag}(3,3,23) \text{ or } \text{diag}(1,3,3,23), \text{ respectively},
\]
and $\ww=\sum_{i=1}^n \dd_i$. Then $Q(\ww)\equiv 5 \text{ or } 6 \text{ (mod }8\text{)}$, respectively, and one may find a representation of $\mathbb{Z}$-cosets from $K+\frac{1}{2}\ww$ to $\Sigma_{13}$ or $\Sigma_{14}$, respectively. However, $K$ cannot be represented by $I_5$ or $I_6$, respectively, over $\mathbb{Q}_3$. Hence, we have $g_\Delta(3)\ge 13$ and $g_\Delta(4)\ge 14$.
\end{proof}

\begin{proof}[Proof of Theorem \ref{thm2}] 
By Propositions \ref{gdelta1} and \ref{lbdgdeltan}, it is enough to prove that $g_\Delta(n)\le n+10$ for each $2\le n \le 4$. 
The proof is a case-by-case analysis according to $n$ and the shape of $\ww$.
For each case, the proof will show how we can determine $g_\Delta(n)\le n+10$.

We assume that $Q(\ww)\equiv k \Mod{8}$ with $1\le k \le 8$ and let $r_{K,\ww}$ be the integer defined in Lemma \ref{lemneccond}.
Also, we assume that the Gram matrix $M$ of $K$ is Minkowski reduced so that $M$ satisfies all conditions given in \eqref{cond1}.
Also, by replacing $\dd_j$ with $\pm \dd_j$ suitably, we may further assume that
\begin{equation} \label{cond2}
m_{1j}\ge0 \quad \text {for any } 2 \le j \le n.
\end{equation}
Under the conditions \eqref{cond1} and \eqref{cond2}, the necessary and sufficient condition for $M$ to be a Minkowski reduced positive definite form is that
\begin{equation}\label{cond3}
m_{ij} \ge -\frac{1}{2}(m_{11}+m_{ii})+m_{1i}+m_{1j} \quad \text{for any } 2\le i < j \le n,
\end{equation}
and when $n=4$,
\begin{equation}\label{cond4}
\begin{array} {rcl}
&& -\frac{1}{2}(m_{22}+m_{33})-m_{23}-m_{24},\\ [0.2em]
&&-\frac{1}{2}(m_{22}+m_{33})+m_{23}+m_{24},\\ [-0.5em]
m_{34}&\ge& \\[-0.5em]
&&-\frac{1}{2}(m_{11}+m_{22}+m_{33})+m_{12}+m_{13}+m_{14}-m_{23}-m_{24},\\[0.2em]
&&-\frac{1}{2}(m_{11}+m_{22}+m_{33})-m_{12}+m_{13}+m_{14}+m_{23}+m_{24},
\end{array}
\end{equation}

\begin{equation}\label{cond5}
\begin{array} {rcl}
&& \frac{1}{2}(m_{22}+m_{33})-m_{23}+m_{24},\\ [0.2em]
&&\frac{1}{2}(m_{22}+m_{33})+m_{23}-m_{24},\\ [-0.5em]
m_{34}&\le& \\[-0.5em]
&&\frac{1}{2}(m_{11}+m_{22}+m_{33})-m_{12}+m_{13}-m_{14}-m_{23}+m_{24},\\[0.2em]
&&\frac{1}{2}(m_{11}+m_{22}+m_{33})-m_{12}-m_{13}+m_{14}+m_{23}-m_{24}.
\end{array}
\end{equation}

\noindent \textbf{(Case 1)} 
We shall prove $g_\Delta(2)\le 12$ by showing $K+\frac{1}{2}\ww$ is represented by $\Sigma_r$ for some $r\le 12$. 
By Lemma \ref{lemneccond} and part (1) of Lemma \ref{lem4}, we may assume that condition \ref{neccond:1} of Lemma \ref{lemneccond} holds and 
\[Q(\ww)>12, \,\,  r_{K,\ww}>12 \text{ and } 1\le k\le 4.\]
\noindent\underline{Case 1-(i)} 
Assume that $\ww=\dd_i$ for $i=1$ or $2$.
From the assumption, we have $m_{ii}>12$.
Hence, we have $Q(\mathbf{x})-kx_i^2$ is positive definite, since by Lemma \ref{lem3},
\[Q(\mathbf{x})-kx_i^2 \ge \frac{3}{4} \text{max} (m_{11}x_1^2,m_{22}x_2^2)-4x_i^2 >0,\]
for any $\mathbf{x}\neq \mathbf{0}$. Therefore, by Lemma \ref{lem4} (2), we conclude that $K+\frac{1}{2}\ww$ is represented by $\Sigma_{k+8}$, where $k+8\le 12$.\\

\noindent\underline{Case 1-(ii)}
Assume that $\ww=\dd_1+\dd_2$. If $m_{22} \ge 6$, then $Q(\mathbf{x})-kx_2^2$ is positive definite by Lemma \ref{lem3}, hence we are done by Lemma \ref{lem4}. 
For any $M$ satisfying $1\le m_{11} \le m_{22} \le 5$ and $0\le m_{12} \le \frac{1}{2}m_{11}$, it does not satisfy the assumption of (Case 1). This proves Case 1.\\

\noindent \textbf{(Case 2)} 
Now we shall prove $g_\Delta(3)\le 13$ by showing $K+\frac{1}{2}\ww$ is represented by $\Sigma_r$ for some $r\le 13$. 
As in Case 1, we may assume that condition \ref{neccond:1} of Lemma \ref{lemneccond} holds and 
\[Q(\ww)>13, \,\,  r_{K,\ww}>13 \text{ and } 1\le k\le 5.\]
\noindent\underline{Case 2-(i)} 
Assume that $\ww=\dd_i$ for $i=1,2$ or $3$.
From the assumption, we have $m_{ii}>13$.
Hence, we have $Q(\mathbf{x})-kx_i^2$ is positive definite, since by Lemma \ref{lem3},
\[ Q(\mathbf{x})-kx_i^2 \ge \frac{1}{2} \text{max} (m_{11}x_1^2,m_{22}x_2^2,m_{33}x_3^2)-5x_i^2 >0,\]
for any $\mathbf{x}\neq \mathbf{0}$. Therefore, by Lemma \ref{lem4} (2), we conclude that $K+\frac{1}{2}\ww$ is represented by $\Sigma_{k+8}$, where $k+8\le 13$.\\

\noindent\underline{Case 2-(ii)} 
Assume that $\ww=\dd_1+\dd_2$. 
If $m_{22}\ge 10$ then $Q(\mathbf{x})-kx_2^2$ is positive definite by Lemma \ref{lem3}, hence we are done by Lemma \ref{lem4} (2).
Now, we may assume that 
\[1\le m_{11} \le m_{22} \le 9 \quad \text{and} \quad 0\le m_{12},m_{13} \le \frac{1}{2}m_{11}.\]
We note that for each triple $(m_{11},m_{22},m_{12})$ satisfying the assumption of Case 2, there exist non-negative integers $k_1,k_2$ such that $k_1+k_2=k$ and $Q(x_1\dd_1+x_2\dd_2)-(k_1x_1^2+k_2x_2^2)$ is positive definite.
Once $m_{11},m_{22},m_{12},m_{13}$ are decided, there are only finitely many candidates of $m_{23}$ by \eqref{cond2} and \eqref{cond3}.

Now, for each fixed $(m_{11},m_{22},m_{12},m_{13},m_{23})$, we do the following process. 
Let $m_{33}$ be the smallest integer greater than or equal to $m_{22}$ satisfying condition \ref{neccond:1} of Lemma \ref{lemneccond}. 
We search for non-negative integers $k_1,k_2$ such that $k_1+k_2=k$ and $Q(\mathbf{x})-(k_1x_1^2+k_2x_2^2)$ is positive definite. 
Once we find such $k_1,k_2$, then we are done by Lemma \ref{lem4} (2). 
Otherwise, we put the matrix $M$ in a list, raise $m_{33}$ by $2$ and then repeat searching for $k_1,k_2$. 
Note that this process ends in a finite number of steps, since the discriminant of the form $Q(\mathbf{x})-(k_1x_1^2+k_2x_2^2)$ is an increasing linear function of $m_{33}$ for each possible pair $(k_1,k_2)$. 
Running this process by a computer program, the final list of matrices obtained is empty.\\

\noindent\underline{Case 2-(iii)}
For the remaining cases, we have $i_t=3$. 
If $m_{33}\ge 10$, then $Q(\mathbf{x})-kx_3^2$ is positive definite by Lemma \ref{lem3}. 
Thus, we are done by Lemma \ref{lem4} (2).
Hence, we are left with finitely many candidates of $M$ all of which have $m_{33}\le 9$.
For each such $M$ satisfying the assumption of Case 2, by a computer program, we can find non-negative integers $k_{i_1},...,k_{i_t}$ such that
\[k_{i_1}+\cdots+k_{i_t}=k \quad \text{and}\quad Q(\mathbf{x})-(k_{i_1}x_{i_1}^2+\cdots+k_{i_t}x_{i_t}^2)\]
is positive definite. Thus, we are done by Lemma \ref{lem4} (2).\\

\noindent \textbf{(Case 3)} 
Lastly, we shall prove $g_\Delta(4)\le 14$ by showing $K+\frac{1}{2}\ww$ is represented by $\Sigma_r$ for some $r\le 14$. 
As before, we may assume that condition \ref{neccond:1} of Lemma \ref{lemneccond} holds and 
\[Q(\ww)>14, \,\,  r_{K,\ww}>14 \text{ and } 1\le k\le 6.\]
\noindent\underline{Case 3-(i)} 
Assume that $i_t=4$, where there are $8$ possible cases.  
If $m_{44}\ge 31$, then by Lemma \ref{lem3}, $Q(\mathbf{x})-kx_4^2$ is positive definite, since
$$ Q(\mathbf{x})-kx_4^2 \ge \frac{1}{5} \text{max} (m_{11}x_1^2,m_{22}x_2^2,m_{33}x_3^2,m_{44}x_4^2)-6x_4^2 >0,$$
for any $\mathbf{x}\neq \mathbf{0}$.
Thus, we are done by Lemma \ref{lem4} (2).
By \eqref{cond1}$-$\eqref{cond5}, we are left with finitely many candidates of $M$ to check.
For each of these $M$ that satisfies the assumption of Case 3, by a computer program, we can find non-negative integers $k_{i_1},...,k_{i_t}$ such that 
\[k_{i_1}+\cdots+k_{i_t}=k \quad \text{and}\quad Q(\mathbf{x})-(k_{i_1}x_{i_1}^2+\cdots+k_{i_t}x_{i_t}^2)\] 
is positive definite, except for the four $\mathbb{Z}$-cosets $M+\frac{1}{2}\ww$, where $\ww=\dd_1+\dd_4$ and $M$ is one of the following matrices:
\[\begin{small}
	\begingroup
\setlength\arraycolsep{3pt}
\begin{pmatrix}
9&3&3&2\\ 
3&9&3&-4\\
3&3&9&-4\\
2&-4&-4&9
\end{pmatrix},
\begin{pmatrix}
9&3&4&2\\
3&9&3&-4\\
4&3&9&-3\\
2&-4&-3&9
\end{pmatrix},
\begin{pmatrix}
9&4&3&2\\
4&9&3&-3\\
3&3&9&-4\\
2&-3&-4&9
\end{pmatrix},
\begin{pmatrix}
9&4&4&2\\
4&9&3&-3\\
4&3&9&-3\\
2&-3&-3&9
\end{pmatrix}.
\endgroup
\end{small}
\]
Note that, two $\mathbb{Z}$-cosets corresponding to the first matrix and the last one with $\ww=\dd_1+\dd_4$ are isometric to each other, and they are represented by $\Sigma_{14}$. Also, the other two matrices also give an equivalent $\mathbb{Z}$-cosets, which are represented by $\Sigma_{14}$. 
This, together with Lemma \ref{lem4} (2), implies the claim in this case. \\

\noindent\underline{Case 3-(ii)} 
Assume that $i_t=3$, where there are $4$ possible cases.  
If $m_{33}\ge 31$, then $Q(\mathbf{x})-kx_3^2$ is positive definite by Lemma \ref{lem3}. Thus, we are done by Lemma \ref{lem4} (2).
Now, we may assume that 
\[1\le m_{11} \le m_{22} \le m_{33} \le 30.\]
Then there are only finitely many candidates of matrix $M-m_{44}E_{44}$. 
Note that, for each candidate that we should concern, there exist non-negative integers $k_{i_1},...,k_{i_t}$ such that 
\[k_{i_1}+\cdots+k_{i_t}=k\quad \text{and}\quad Q(x_1\dd_1+x_2\dd_2+x_3\dd_3)-(k_{i_1}x_{i_1}^2+\cdots+k_{i_t}x_{i_t}^2)\]
is positive definite.
We can run a process similar to the one described in the Case 2-(ii).
However this time we have four matrices on the final list and they appear when $\ww=\dd_1+\dd_3$. 
Each of the corresponding $\mathbb{Z}$-cosets is isometric to one of the $\mathbb{Z}$-cosets described in Case 3-(i). Hence we proves the claim.\\

\noindent\underline{Case 3-(iii)} 
Assume that $i_t=2$. 
By a similar argument as before, we may assume that $1\le m_{11} \le m_{22}\le 30$. Then there are only finitely many candidates of $7$-tuple $(m_{11},m_{12},m_{13},m_{14},m_{22},m_{23},m_{24})$.
For each of these $7$-tuples that satisfies the assumption of Case 3, we can check that $Q(x_1\dd_1+x_2\dd_2)-kx_2^2$ is positive definite and $C:=m_{11}-\frac{5m_{12}^2}{2(m_{22}-k)}>0$.
Let $m_{22}'=m_{22}-k$ and put
\[
C_0:=\text{max}\left( \frac{10}{3}\cdot\frac{m_{23}}{m_{22}'}+\frac{m_{13}(m_{13}+m_{14})}{C} \begin{matrix} \\,\end{matrix}\,\, \frac{10}{3}\cdot \frac{m_{24}}{m_{22}'}+\frac{m_{14}(m_{13}+m_{14})}{C} \right).
\]
If $m_{33}>\frac{8}{3}C_0$, then $Q(\mathbf{x})-kx_2^2 $ is greater than or equal to 
\begin{equation*}
\begin{aligned}[t]
\frac{2m_{22}'}{5}\left(\!x_2+\frac{5}{2}\frac{m_{12}}{m_{22}'}x_1\!\right)^2 \!\! +
\frac{3m_{22}'}{10}\left(\!x_2+\frac{10}{3}\frac{m_{13}}{m_{22}'}x_3\!\right)^2 \!\! +
\frac{3m_{22}'}{10}\left(\!x_2+\frac{10}{3}\frac{m_{14}}{m_{22}'}x_4\!\right)^2\\[3pt]
+C\cdot\frac{m_{13}}{m_{13}+m_{14}}\left(x_1+\frac{m_{13}}{C}x_3\right)^2+
C\cdot\frac{m_{14}}{m_{13}+m_{14}}\left(x_1+\frac{m_{14}}{C}x_4\right)^2\\[3pt]
+Q(x_3\dd_3+x_4\dd_4)-C_0(x_3^2+x_4^2).
\end{aligned}
\end{equation*}
Since $m_{44}\ge m_{33} > \frac{8}{3}C_0$ and $Q(x_3\dd_3+x_4\dd_4)\ge \frac{3}{8}(m_{33}x_3^2+m_{44}x_4^2)$ by Lemma \ref{lem3}, we conclude that $Q(\mathbf{x})-kx_2^2$ is positive definite.
Hence, we are done by Lemma \ref{lem4} (2).

Now, we may assume that $m_{33}\le \frac{8}{3}C_0$, so that there are only finitely many candidates of matrix $M-m_{44}E_{44}$.
We note that the quadratic form $Q(x_1\dd_1+x_2\dd_2+x_3\dd_3)-kx_2^2$ is positive definite for each of these candidates.
We run the same process as described in the Case 3-(ii) and obtain a list of four matrices $M$ which appear only when $\ww=\dd_1+\dd_2$.
Each of the corresponding $\mathbb{Z}$-cosets is isometric to one of the $\mathbb{Z}$-cosets described in Case 3-(i). Hence we proves the claim.\\

\noindent\underline{Case 3-(iv)} 
Finally, we assume that $\ww=\dd_1$. 
One may check by Lemma \ref{lem3} that $Q(\mathbf{x})-kx_1^2$ is positive definite except for $m_{11}\in \{20,21,22,30\}$. Hence, by Lemma \ref{lem4} (2), we may assume that $m_{11}\in \{20,21,22,30\}$, so that there are only finitely many candidates of $4$-tuple $(m_{11},m_{12},m_{13},m_{14})$.

Put $m_{11}'=m_{11}-k$ and $m_{234}=m_{12}+m_{13}+m_{14}$. If $m_{234}=0$, that is, $m_{12}=m_{13}=m_{14}=0$, then $Q(\mathbf{x})-kx_1^2$ is positive definite obviously. Otherwise, we note that
\[
\begin{array}{rcl}
Q(\mathbf{x})-kx_1^2&=& \displaystyle\sum_{j=2}^4 m_{11}'\frac{m_{12}}{m_{234}}\left(x_1+\frac{m_{234}}{m_{11}'}x_2\right)^2  \\
&&+\,\,Q(x_2\dd_2+x_3\dd_3+x_4\dd_4)- \dfrac{m_{234}}{m_{11}'}\cdot \displaystyle\sum_{j=2}^4 m_{1j}x_j^2.
\end{array}
\]
Hence, if $m_{22}>2\cdot\frac{m_{234}^2}{m_{11}'}$, then $Q(\mathbf{x})-kx_1^2$ is positive definite, since
\[Q(x_2\dd_2+x_3\dd_3+x_4\dd_4)\ge \sum_{j=2}^4 \frac{1}{2}\frac{m_{1j}}{m_{234}} m_{jj}x_j^2 >\dfrac{m_{234}}{m_{11}'}\cdot \sum_{j=2}^4 m_{1j}x_j^2,\]
\noindent by Lemma \ref{lem3}.
Thus, we may assume that $m_{22}\le 2\cdot\frac{m_{234}^2}{m_{11}'}$ and there are only finitely many candidates of $7$-tuple $(m_{11},m_{12},m_{13},m_{14},m_{22},m_{23},m_{24})$.
By a similar argument used in Case 3-(iii), we may further assume that $m_{33}$ is bounded and run the process as described in the other cases.
This time the final list is empty and so we are done.
\end{proof}

\begin{rmk}
One may naturally expect that $g_\Delta(5)=15$. 
However, if we consider a $\mathbb{Z}$-coset $K+\frac{1}{2}\dd_5$, where $K=\mathbb{Z}[\dd_1,...,\dd_5]$ whose Gram matrix with respect to $\{\dd_1,...,\dd_5\}$ is a diagonal matrix $\text{diag}( 2,2,2,2,16)$, then one may verify that 
$$K+\frac{1}{2}\dd_5 \ra \Sigma_{16}, \quad \text{but} \quad K+\frac{1}{2}\dd_5 \nra \Sigma_8,$$
which implies that $g_\Delta(5)\ge16$. 
\end{rmk}

\subsection*{Acknowledgments}
I would like to express my gratitude to Professor Byeong-Kweon Oh who is my supervisor and Professor Wai Kiu Chan for their valuable advice, and to the referee for carefully reading this paper and making many helpful comments.

\end{document}